\documentclass[11pt,oneside,pagesize=auto]{scrartcl}

\newcommand{\titl}{A Closer Look at Covering Number Bounds for Gaussian Kernels}

\newcommand{\keyw}{
Gaussian kernels, Covering numbers
}
\usepackage{tocbasic}
\usepackage[draft=false]{scrlayer-scrpage}
\usepackage[onehalfspacing]{setspace}

\KOMAoptions{mpinclude=false}
\KOMAoptions{headsepline=false}
\KOMAoptions{footsepline=false}

\KOMAoptions{titlepage=false}

\KOMAoptions{bibliography=totoc}

\KOMAoptions{abstract=true}

\KOMAoptions{BCOR=0cm, DIV=12}

\addtokomafont{title}{\boldmath}
\addtokomafont{section}{\boldmath}
\addtokomafont{subsection}{\boldmath}
\addtokomafont{paragraph}{\boldmath}
\addtokomafont{sectionentry}{\boldmath}

\usepackage[latin1]{inputenc}
\usepackage[T1]{fontenc}
\usepackage{lmodern}
\usepackage[english]{babel}
\usepackage{microtype}
\usepackage{enumitem}

\usepackage{amsmath}
\usepackage{amssymb}
\usepackage{amsthm}
\usepackage{mathrsfs}
\usepackage{mathtools}
\usepackage{bbm}
\usepackage{bm}

\usepackage[numbers,square]{natbib}

\usepackage{xcolor}

\usepackage[
pdfusetitle, 
pdfkeywords={\keyw}, 
backref = false, 
hidelinks,
pdfpagelabels = true, 
pdfstartview = FitH, 
bookmarksopen = true, 
bookmarksnumbered = true, 
plainpages = false, 
hypertexnames = false
] {hyperref}

\setlist{itemsep=0pt}
\setenumerate[1]{label={\roman*).}, ref={\textit{\roman*)}}}
\setenumerate[2]{label={(\alph*)}}

\newtheoremstyle{famous} 
{} 
{} 
{\itshape}
{} 
{\bfseries\sffamily\boldmath} 
{} 
{\newline} 
{\boldmath\thmnumber{\begin{footnotesize}#2\end{footnotesize} }\thmnote{#3}}

\newtheoremstyle{kursiv}
{} 
{} 
{\itshape} 
{} 
{\bfseries\sffamily\boldmath} 
{} 
{ } 
{\boldmath\thmnumber{\begin{footnotesize}#2\end{footnotesize} }\thmname{#1 }\thmnote{(#3)}}

\newtheoremstyle{normal}
{} 
{} 
{\rmfamily} 
{} 
{\bfseries\sffamily\boldmath} 
{} 
{ } 
{\boldmath\thmnumber{\begin{footnotesize}#2\end{footnotesize} }\thmname{#1 }\thmnote{(#3)}}

\newtheoremstyle{oNum}
{} 
{} 
{\rmfamily} 
{} 
{\bfseries\sffamily\boldmath} 
{} 
{\newline} 
{\thmname{#1 }\thmnote{(#3)}}

\makeatletter
\def\@endtheorem{\endtrivlist}
\makeatother

\theoremstyle{kursiv}
\newtheorem{thm}{Theorem}[section]

\newtheorem{lem}[thm]{Lemma}

\theoremstyle{normal}

\theoremstyle{famous}

\DeclareFontFamily{U}{matha}{\hyphenchar\font45}
\DeclareFontShape{U}{matha}{m}{n}{
      <5> <6> <7> <8> <9> <10> gen * matha
      <10.95> matha10 <12> <14.4> <17.28> <20.74> <24.88> matha12
      }{}
\DeclareSymbolFont{matha}{U}{matha}{m}{n}
\DeclareFontFamily{U}{mathx}{\hyphenchar\font45}
\DeclareFontShape{U}{mathx}{m}{n}{
      <5> <6> <7> <8> <9> <10>
      <10.95> <12> <14.4> <17.28> <20.74> <24.88>
      mathx10
      }{}
\DeclareSymbolFont{mathx}{U}{mathx}{m}{n}
\DeclareMathDelimiter{\vvvert}{0}{matha}{"7E}{mathx}{"17}

\DeclareFontFamily{U}{mathx}{\hyphenchar\font45}
\DeclareFontShape{U}{mathx}{m}{n}{
      <5> <6> <7> <8> <9> <10>
      <10.95> <12> <14.4> <17.28> <20.74> <24.88>
      mathx10
      }{}
\DeclareSymbolFont{mathx}{U}{mathx}{m}{n}
\DeclareFontSubstitution{U}{mathx}{m}{n}
\DeclareMathAccent{\widecheck}{0}{mathx}{"71}
\DeclareMathAccent{\wideparen}{0}{mathx}{"75}

\DeclareNewTOC[%
type=todo,%
types=todos,%
name=Todo,%
listname={Todo Liste}%
]{lotodo}
\setuptoc{lotodo}{chapteratlist} 

\newcounter{todono}
\renewcommand{\thetodono}{(\arabic{todono})}

\newcommand{\todo}[2]{\stepcounter{todono}
$^{\text{\thetodono}}$\marginline{
\begin{singlespace}
\footnotesize
\textbf{{\tiny\thetodono} #1}\par #2
\end{singlespace}}%
\addcontentsline{lotodo}{table}{#1 (#2)}}

\newcommand{\R}{{\mathbb R}}

\newcommand{\N}{{\mathbb N}}

\DeclareMathOperator{\Id}{Id}

\DeclareMathOperator{\spanning}{span}

\DeclareMathOperator{\rank}{rank}

\newcommand{\lbracket}[2][]{
\def\tempsize{#1}
\ifx\tempsize\empty #2
\else\csname #1l\endcsname{#2}
\fi}

\newcommand{\rbracket}[2][]{
\def\tempsize{#1}
\ifx\tempsize\empty #2
\else\csname #1r\endcsname{#2}
\fi}

\def\savedexp{} 

\def\newfunction#1{%

\expandafter\def\csname#1\endcsname{\expandafter\futurelet\expandafter\next\csname#1temp\endcsname}
\expandafter\def\csname#1temp\endcsname{%
\ifx^\next \expandafter\expandafter\csname#1E\endcsname
\else \expandafter\expandafter\csname#1N\endcsname
\fi}

\expandafter\def\csname#1E\endcsname^##1{
\def\savedexp{{##1}}%
\csname#1X\endcsname}

\expandafter\def\csname#1N\endcsname{
\def\savedexp{}%
\csname#1X\endcsname}
}

\newcommand{\cuball}[1]{\overline{B}_{#1}}
\newcommand{\ouball}[1]{\mathring{B}_{#1}}

\newfunction{entropy}
\newcommand{\entropyX}[3][]{
\varepsilon_{#2}
\ifx\savedexp\empty
\else
^{\savedexp}
\fi 
\lbracket[#1](#3\rbracket[#1])
}

\newfunction{entropyd}
\newcommand{\entropydX}[3][]{
e_{#2}
\ifx\savedexp\empty
\else
^{\savedexp}
\fi 
\lbracket[#1](#3\rbracket[#1])
}

\newfunction{covering}
\newcommand{\coveringX}[3][]{
\mathcal{N}
\ifx\savedexp\empty
\else
^{\savedexp}
\fi 
\lbracket[#1](#3,#2\rbracket[#1])
}

\newfunction{coveringi}
\newcommand{\coveringiX}[3][]{
\mathcal{N}_0
\ifx\savedexp\empty
\else
^{\savedexp}
\fi 
\lbracket[#1](#3,#2\rbracket[#1])
}

\newfunction{packing}
\newcommand{\packingX}[3][]{
\mathcal{P}
\ifx\savedexp\empty
\else
^{\savedexp}
\fi 
\lbracket[#1](#3,#2\rbracket[#1])
}

\newfunction{coveringl}
\newcommand{\coveringlX}[3][]{
\mathcal{H}
\ifx\savedexp\empty
\else
^{\savedexp}
\fi 
\lbracket[#1](#3,#2\rbracket[#1])
}

\newcommand{\sfrac}[2]{{#1}/{#2}}

\newcommand{\eqspace}{\,}

\newcommand{\emb}[3][]{
I_{#2}
\ifx #3\empty%
\else%
\lbracket[#1][#3\rbracket[#1]]%
\fi}

\let\cuball=\uball

\newcommand{\lamWo}{W}      

\newcommand{\ef}[1]{p_{#1}}       
\newcommand{\efi}[1]{p_{#1}^{-1}} 
\newcommand{\efii}[1]{h_{#1}}     

\newcommand{\qs}{q_{\sigma}}
\newcommand{\vecsigma}{{\bm{\sigma}}}

\newcommand{\ada}[1]{#1.}

\renewcommand{\todo}[2]{} 

\renewcommand{\thanks}[1]{}
\renewcommand{\textsuperscript}[1]{}

\begin{document}


\title{\titl\texorpdfstring{\thanks{This research did not receive any specific grant from funding agencies in the public, commercial, or not-for-profit sectors.}}{}}

\author{%
Ingo Steinwart\texorpdfstring{\textsuperscript{a}}{} and 
Simon Fischer\texorpdfstring{\textsuperscript{a,}\thanks{Corresponding author}}{}}

\publishers{\normalsize \textsuperscript{a}%
Institute for Stochastics and Applications\\
Faculty 8: Mathematics and Physics\\
University of Stuttgart\\
70569 Stuttgart Germany \\
\texttt{\small $\{$ingo.steinwart, simon.fischer$\}$@mathematik.uni-stuttgart.de}
}

\date{\normalsize\today}

\maketitle

\begin{abstract}

We establish some new bounds on the log-covering numbers of (anisotropic) Gaussian
reproducing kernel Hilbert spaces. Unlike previous results in this direction we focus on
small explicit constants and their dependency on crucial parameters such as the kernel bandwidth 
and the size and dimension of the underlying space.

\end{abstract} 






\section{Introduction}


Gaussian kernels and their reproducing kernel Hilbert spaces (RKHSs) play a central role for kernel-based learning algorithms such as support vector machines (SVMs), see e.g.\ \cite{ScSm2001,CuZh2007,StCh2008}, and Gaussian processes for machine learning, see e.g.\ \cite{RaWi2006,KaHeSeSr2018}. For the analysis of such learning algorithms one usually needs to bound both the \emph{approximation error}, which quantitatively describes how well the considered RKHS approximates certain classes of smooth functions, and the \emph{estimation error}, which bounds the uncertainty caused by the statistical nature of the observations the algorithm learns from. Moreover, the estimation error is typically analyzed with the help of bounds on certain entropy- or covering numbers of the involved RKHSs, and in the case of Gaussian RKHSs these bounds crucially depend on quantities such as the considered kernel width. The major focus of this work is to analyze this dependence.
To be more precise, recall that the (isotropic) Gaussian kernels are given by 
\begin{equation}\label{eq:intro:gaussian_isotropic}
k_\sigma(x,x') \coloneqq \exp\bigl(- \sigma^2\| x - x'\|_{\ell_2^d}^2\bigr)
\eqspace, \qquad \qquad x,x'\in X,
\end{equation}
where $X$ is a subset of $\R^d$, $\sigma >0$ is the so-called kernel width, and $\|\cdot\|_{\ell_2^d}$ denotes the Euclidean norm on $\R^d$. Moreover, we write $H_\sigma(X)$ for the corresponding RKHS, see \cite[Chapter 4]{StCh2008} for details about Gaussian RKHSs as well as for a general introduction to RKHSs.

In order to underpin the importance of log-covering number bounds with explicit and well-understood constants we will briefly sketch the analysis of SVMs using the least squares loss and the Gaussian kernel in the following. To this end, let us recall that the covering numbers of a bounded subset $A\subseteq E$ of some Banach space $E$ are defined by 
\begin{equation*}
\covering{\varepsilon}{A}\coloneqq\min\Bigl\{n\in\N:\ \exists x_1,\ldots,x_n\in E:\ A\subseteq \bigcup_{i=1}^n x_i + \varepsilon\cuball{E}\Bigr\}
\eqspace, \qquad \qquad \varepsilon>0,
\end{equation*}
where $\cuball{E}$ denotes the closed unit ball of $E$. Moreover, for a bounded linear operator $T:E\to F$ between two Banach spaces $E$ and $F$, the log-covering numbers are $\coveringl{\varepsilon}{T}\coloneqq\log(\covering{\varepsilon}{T\cuball{E}})$. Now, if $X\subseteq \R^d$ is bounded it is well-known that for all $\sigma\geq 1$ and $p\in (0,1)$ we have a constant $K_{X,\sigma,p}$ such that 
\begin{equation}\label{eq:intro:example_upper_bound}
\coveringl[big]{\varepsilon}{\Id:H_\sigma(X) \to \ell_\infty(X)} \leq K_{X,\sigma,p} \cdot \varepsilon^{-p}
\eqspace, \qquad \qquad \varepsilon\in (0,1],
\end{equation}
where $\Id : H_\sigma(X) \to \ell_\infty(X)$ denotes the canonical embedding of $H_\sigma(X)$ into the space $\ell_\infty(X)$ of bounded functions $f:X\to \R$ equipped with the usual supremum norm $\|\cdot\|_\infty$, see e.g.\ \cite[Theorem~6.27 and Exercise~6.8]{StCh2008}. However, the dependency of the constant $K_{X,\sigma,p}$ on $X$, $\sigma$, and $p$ is far from being well-understood.

Let us now briefly describe how this dependency influences the learning performance guarantees of SVMs using the least squares loss and a Gaussian kernel $k_\sigma$. To this recall that for a dataset $D= ((x_1,y_1),\ldots,(x_n,y_n) )\in(X\times[-1,1])^n$ of length $n$ and a regularization parameter $\lambda >0$, such an SVM produces a  decision function $f_{D,\lambda,\sigma}:X\to[-1,1]$ that minimizes some regularized empirical error quantity over $H_\sigma(X)$, in order to recover the true but unknown regression function $f^\ast:X\to[-1,1]$. In this scenario \cite[Theorem 7.23]{StCh2008} in combination with \eqref{eq:intro:example_upper_bound} gives, for every $f_0\in H_\sigma(X)$, the following over all error bound 
\begin{equation}\label{eq:intro:oi:approximation}
\|{f}_{D,\lambda,\sigma} - f^\ast\|_{L_2(\nu)}^2
\leq 9 \bigl(\lambda \|f_0\|_{H_\sigma(X)}^2 + \|f_0 - f^\ast\|_{L_2(\nu)}^2\bigr) 
 + C_p \cdot \frac {K_{X,\sigma,p}}{\lambda^{\sfrac{p}{2}} \, n}
+ \tilde{\epsilon}(n, \lambda, \tau, \sigma, p, f_0)
\eqspace,
\end{equation}
which holds true with probability not less than $1 - 3 e^{-\tau}$. Here, $C_p$ is a constant whose dependency on $p$ is explicitly given, and $\tilde{\epsilon}(n, \lambda, \tau, \sigma, p, f_0)$ is an additional error term, that for common choices of $f_0$, $\lambda$, $p$, $\tau$, and $\sigma$ is dominated by the term
\begin{equation*}
\epsilon(n,\lambda,\sigma,p)\coloneqq
C_p \cdot \frac {K_{X,\sigma,p}}{\lambda^{\sfrac{p}{2}} \, n}
\eqspace,
\end{equation*}
which in the following we call \emph{estimation error.} 
Together with $\tilde \epsilon(n, \lambda, \tau, \sigma, p, f_0)$, the estimation error bounds the error
caused by statistical fluctuations.
In contrast, the first error term in 
\eqref{eq:intro:oi:approximation}, which does not depend on the sample size $n$, refers to the 
\emph{approximation error}. 

It is well-known that the Gaussian RKHS $H_\sigma(X)$ only contains $C^\infty$-functions and that $H_\sigma(X)$ is dense in $L_p(\nu)$ for all $p\in [1,\infty)$ and all finite measures $\nu$ on $X$. Moreover, if $X$ is compact, then $H_\sigma(X)$ is also dense in $C(X)$. Again, we refer to \cite[Chapter 4]{StCh2008} for details. Now recall that these denseness results guarantee, for example,  that   the \emph{minimal approximation error}
\begin{equation*}
A(\lambda, \sigma,f^\ast) 
\coloneqq \inf\Bigl\{ \lambda \|f_0\|_{H_\sigma(X)}^2 + \|f_0-f^\ast\|_{L_2(\nu)}^2\, :\, f_0\in H_\sigma(X)\Bigr\}
\end{equation*}
satisfies $A(\lambda, \sigma,f^\ast) \to 0$ for $\lambda \to 0$ and \emph{fixed} $\sigma>0$ and $f^\ast\in L_2(\nu)$. For $n\to\infty$ and $\lambda\to 0$ with $\lambda^{\sfrac{p}{2}}n \to \infty$ this shows that \eqref{eq:intro:oi:approximation} vanishes, i.e.\ SVMs using the least squares loss and a \emph{fixed} Gaussian kernel can learn in a purely asymptotic sense, see e.g.\ \cite[Chapters 5, 6, and 9]{StCh2008} for details. However, a more detailed analysis that includes convergence rates for the learning process, requires convergence rates for the approximation error and especially for $A(\lambda, \sigma,f^\ast) \to 0$. Unfortunately, it has been shown in \cite{SmZh2002} that for fixed $\sigma>0$ any polynomial rate even for the minimal approximation error $A(\lambda, \sigma,f^\ast) \to 0$ is impossible if $f\not\in C^\infty$, and the latter is an unacceptable restriction from a learning theoretical point of view.

To address this issue and to be better aligned with empirical knowledge that strongly suggest to vary the width $\sigma$ with the data set, one usually investigates the learning behavior in cases in which we have $\lambda \to 0$ and $\sigma \to \infty$ simultaneously. For example, \cite{EbSt2013} shows that for specific combinations of rates for $\lambda \to 0$ and $\sigma \to \infty$ the approximation error converges to 0 with a polynomial speed whenever $f^\ast$ is contained in some Besov space. 

While this approach solves the issues regarding the approximation error, it simultaneously makes the analysis of the estimation error $\epsilon(n,\lambda,\sigma,p)$ more complicated. To be more precise, for fixed $\sigma$ and $p$ the dependence of $K_{X,\sigma,p}$ on these parameters have no influence on the learning rate, however, if we consider $\sigma \to \infty$ the behavior of $K_{X,\sigma,p}$ plays a crucial role for the estimation error. Since it has been recently observed in \cite{FaSt2018} that the learning rates can be further improved, if we additionally let $p=p_n\to 0$ sufficiently slowly, also the dependence of $K_{X,\sigma,p}$ on $p$ is of interest from a learning theoretical point of view. Moreover, the guarantees on the learning performance obviously become better, if, in addition, $K_{X,\sigma,p}$ only depends on \emph{small} universal constants. Therefore, the goal of this work is to derive bounds on $\coveringl{\varepsilon}{\Id:H_\sigma(X)\to\ell_\infty(X)}$ that do not only have a desirable behavior for $\varepsilon\to 0$, but for which we can also control the behavior of the corresponding constants in $\sigma$, $X$, $d$, and if applicable, in $p$.

To this end, we first refine the analysis of \cite{K2011} by carefully controlling the arising constants. It turns out that the final constants have both small absolute values and a reasonable behavior in the dimension $d$. Unfortunately, however, their behavior for $\sigma\to \infty$ is far from being optimal. For this reason, we present another result that relates the log-covering numbers of $\Id:H_\sigma(X)\to\ell_\infty(X)$ to the log-covering numbers of $\Id:H_1(B_2^d)\to\ell_\infty(B_2^d)$, where $B_2^d\subseteq\R^d$ denotes the closed Euclidean unit ball, and to the covering numbers of the underlying space $X$. As a consequence, we do not only obtain a much better behavior for $\sigma\to \infty$, but also log-covering number bounds for \emph{anisotropic} Gaussian kernels, which are defined by
\begin{equation}\label{eq:intro:gaussian_ard}
k_\vecsigma(x,x') \coloneqq \exp\bigl(-\|D_\vecsigma x - D_\vecsigma x'\|_{\ell_2^d}^2\bigr)
\eqspace, \qquad \qquad x,x'\in X,
\end{equation}
where $D_\vecsigma(x_1,\ldots, x_d) \coloneqq (\sigma_1 x_1,\ldots, \sigma_d x_d)$. Note that these kernels are an example of so-called automatic relevance determination (ARD) kernels, which are particularly popular in the Gaussian processes for machine learning context, see e.g.\ \cite[Chapter 5]{RaWi2006}.

The rest of this work is organized as follows: In the next section we present our main results, discuss their consequences, and compare them to results previously obtained in the literature such as \cite{Zh2002,Zh2003,K2011}. All proofs can be found in Section~\ref{sec:proofs}.


\section{Main Results}


This section contains all main results of this work:
In the first subsection we derive bounds on the log-covering numbers of the embedding $\Id: H_\sigma(B_2^d)\to\ell_\infty(B_2^d)$ of the isotropic Gaussian RKHS defined in \eqref{eq:intro:gaussian_isotropic}. In the second subsection we then show how to generalize these bounds to anisotropic Gaussian kernels \eqref{eq:intro:gaussian_ard} on general bounded sets $X\subseteq \R^d$.

\subsection{Isotropic Gaussian Kernels}\label{sec:isotropic}

Before we present the results of this subsection, let us introduce some notation:
if two functions $f,g:(0,\infty)\to(0,\infty)$ satisfy $\lim_{t\to\infty}\sfrac{f(t)}{g(t)} = 1$ we write $f(t)\sim g(t)$ for $t\to\infty$.
Moreover, 
recall that, 
for $k\in\N$ and $t>0$, the
\emph{generalized binomial coefficient} is defined by
\begin{equation*}
\binom{t}{k} \coloneqq \frac{1}{k!}\prod_{i=1}^k(t-k + i)
\eqspace.
\end{equation*}
Note that for $t\in\N$ this definition coincides with the classical definition of binomial coefficients. In the following, generalized binomial coefficients mainly appear in the form
\begin{equation}\label{eq:results:binom}
\binom{t + d}{d} = \frac{1}{d!}\prod_{i=1}^d(t + i)
\end{equation}
where $d\geq 1$ is an integer and $t>0$. Then the
functions $t\mapsto \binom{t + d}{d}$ and $d\mapsto\binom{t+d}{d}$ are increasing and 
the functions $t\mapsto\binom{t + d}{d}\cdot t^{-d}$ and 
$d\mapsto \binom{t+d}{d}\cdot d^{-t}$ are decreasing with 
\begin{equation*}
\binom{t+d}{d} \sim \frac{t^d}{d!}\quad\text{for }t\to\infty
\quad\text{and} \quad
\binom{t+d}{d} \sim \frac{d^t}{\Gamma(t+1)}\quad\text{for }d\to\infty
\eqspace,
\end{equation*}
where $\Gamma$ denotes the Gamma function.
See Lemma~\ref{lem:entropy:binom_behavior} for the non-obvious assertions.
With these preparations our first result reads as follows.

\begin{thm}\label{thm:results:entropy-1}
For all $d\geq 1$, all $\sigma>0$, and all $0<\varepsilon\leq 1$ we have
\begin{equation*}
\coveringl[big]{\varepsilon}{\Id: H_\sigma(B_2^d)\to\ell_\infty(B_2^d)} 
\leq 
\binom{2 e (1+\sigma^2) + d}{d} \cdot e^{-d}\cdot \frac{\log^{d+1}(\sfrac{4}{\varepsilon})}{\log\log^{d}(\sfrac{4}{\varepsilon})}
\eqspace.
\end{equation*}
\end{thm}

Note that Theorem~\ref{thm:results:entropy-1} recovers the asymptotic behavior of 
$\varepsilon\mapsto \coveringl{\varepsilon}{\Id: H_\sigma(B_2^d)\to\ell_\infty(B_2^d)}$ found by Kühn in \cite{K2011},
which in turn improved the earlier results in \cite{Zh2002,Zh2003}.
By presenting a corresponding 
lower bound on the log-covering numbers, 
\cite{K2011} further shows that this behavior in $\varepsilon$ is optimal.
Unlike the upper bound
in \cite{K2011}, however, Theorem~\ref{thm:results:entropy-1} also provides an upper bound for the behavior in 
$\sigma$ and $d$ that is expressed by the constant
\begin{equation*}
 K_{d,\sigma} \coloneqq \binom{2 e(1+\sigma^2) + d}{d} \cdot e^{-d}
\eqspace.
\end{equation*}

To better understand the behavior of this constant in $d$, let us first consider the case $\sigma = 1$.
Since 
$d\mapsto \binom{4 e + d}{d} \cdot d^{-4 e}$
is decreasing as mentioned above we then find 
\begin{equation*}
C_d
\coloneqq \frac{K_{d,1}}{d^{4 e} e^{-d}} 
\leq \binom{4 e + 1}{1} \cdot 1^{-4 e}
= 4 e + 1 
\approx 11.8731
\end{equation*}
for all $d\geq 1$. Moreover, some numerical calculations show that for $d=1$ we have 
$K_{1,1}=4+\sfrac{1}{e}\approx 4.3679$, 
while for $d\geq 2$ we have 
$C_d\leq C_2 \approx 0.0407\leq 0.05$,
and hence we obtain 
\begin{equation}\label{kd1-bound}
K_{d,1} \leq 0.05 \cdot d^{4 e} e^{-d}
\eqspace, \qquad \qquad d\geq 2.
\end{equation}
In this respect note that \cite[Proposition~1]{Zh2002} found a constant behaving like $d^{d+1}$ for a $\log^{d+1}(\sfrac{1}{\varepsilon})$-type bound on the log-covering numbers. Unfortunately, this result is not directly comparable to \eqref{kd1-bound} since \cite{Zh2002} considered the set $X=[0,1]^d$. Since $[0,1]^d\subseteq\sfrac{1}{2} + \sfrac{\sqrt{d}}{2} \cdot B_2^d$ combining \eqref{kd1-bound} with the later established \eqref{kd1-bound-2} for $r = \sfrac{\sqrt{d}}{2}$ and $\sigma=1$, however, we obtain a constant not exceeding $0.05 \cdot d^{4 e} (\sfrac{2 e}{3})^{-d} d^{\sfrac{d}{2}}$ for $X= [0,1]^d$ and $d\geq 2$. In other words, our analysis does improve the above mentioned results of \cite{Zh2002} in both $\varepsilon$ and $d$.

Furthermore, this inequality correctly describes the asymptotic behavior of $K_{d,1}$ for $d\to \infty$, since 
our considerations at the beginning of this section show
\begin{equation*}
\lim_{d\to \infty} C_d
= \lim_{d\to \infty}\binom{4 e + d}{d}\cdot d^{-4 e}
= \frac{1}{\Gamma(4 e + 1)}
\approx 3.4130\cdot 10^{-8}
\eqspace.
\end{equation*}
Finally, some additional numerical calculations give $K_{d,1} \leq 30$ for all $d\geq 1$,
and the maximal value of $K_{d,1}$ is attained at $d=6$.

Let us now consider the behavior of $K_{d,\sigma}$ in $\sigma$ for a fixed $d\geq 1$. 
To this end, we first observe that $K_{d,\sigma}$ is increasing in $\sigma$, and hence we have 
$K_{d,\sigma}>K_{d,0}=\binom{2 e + d}{d} e^{-d}>0$ for all $\sigma>0$. Moreover, 
the representation in \eqref{eq:results:binom} directly gives 
\begin{equation}\label{eq:results:sigma}
\frac{2^d}{d!} (1+\sigma^2)^d
\leq K_{d,\sigma}
\leq \frac{4^d}{d!} (1+\sigma^2)^d
\end{equation}
for all $\sigma>0$ satisfying $2 e (1+\sigma^2) \geq d$. Consequently 
the constant $K_{d,\sigma}$ grows like $\sigma^{2d}$ for $\sigma\to\infty$,
compared to the $\sigma^{2d+2}$-behavior of the already discussed result in \cite{Zh2002}.
Below in Section~\ref{sec:anisotropic} we will see that we can find another constant for the estimate of Theorem~\ref{thm:results:entropy-1} that only grows like $\sigma^d$ for $\sigma\to\infty$.

Our next goal is to show that the size of the constant in Theorem~\ref{thm:results:entropy-1} is significantly 
influenced by the choice of the considered range of $\varepsilon$. More precisely, Theorem~\ref{thm:results:entropy-1} considers the 
maximal range $0<\varepsilon\leq 1$, since we have 
$\|\Id: H_\sigma(B_2^d)\to\ell_\infty(B_2^d)\|=1$, and thus we find 
\begin{equation*}
\coveringl[big]{\varepsilon}{\Id: H_\sigma(B_2^d)\to\ell_\infty(B_2^d)}=0
\eqspace, \qquad \qquad \varepsilon\geq 1.
\end{equation*}
Our next theorem shows that by considering a smaller range for $\varepsilon$, we can substantially decrease the constant 
appearing in the estimate.
For its formulation, we recall that Lambert's $W$-function is the inverse of $t\mapsto te^t$. Note 
that on $(-\sfrac{1}{e},0)$ the inverse is multi-valued and throughout this work 
we use the upper branch $\lamWo:[-\sfrac{1}{e},\infty) \to [-1,\infty)$, which is often denoted by $W_0$ in the literature. Finally, 
recall that $\lamWo$ is increasing and 
$\lamWo(t)\sim\log(t)$ for $t\to\infty$.
Now our second result reads as follows.

\begin{thm}\label{thm:results:entropy-2}
For all $d\geq 1$, all $\sigma > 0$, and 
$0<\varepsilon_0 \leq 4\exp(-e^{1 + \sigma^{-2}})$ we define $y_0 \coloneqq \log(\sfrac{4}{\varepsilon_0})$, $x_0\coloneqq\sfrac{2 y_0}{\lamWo(\frac{y_0}{e\sigma^2})}$, and 
\begin{equation*}
K_{d,\sigma,\varepsilon_0} \coloneqq \binom{x_0 + d}{d} \cdot \Bigl(\frac{\log(y_0)}{y_0}\Bigr)^d
\eqspace.
\end{equation*}
Then for all $0<\varepsilon\leq \varepsilon_0$ we have 
\begin{equation*}
\coveringl[big]{\varepsilon}{\Id: H_\sigma(B_2^d)\to\ell_\infty(B_2^d)} 
\leq K_{d,\sigma,\varepsilon_0} 
\cdot \frac{\log^{d+1}(\sfrac{4}{\varepsilon})}{\log\log^{d}(\sfrac{4}{\varepsilon})}
\eqspace.
\end{equation*}
\end{thm}

To appreciate Theorem~\ref{thm:results:entropy-2} we note that for $\varepsilon_0\to 0$ we have 
$y_0 \to \infty$ and $x_0\to\infty$. Since $\lamWo(t)\sim\log(t)$ and $\binom{t+d}{d} \sim \sfrac{t^{d}}{d!}$ for $t\to\infty$ we then find
\begin{equation*}
\lim_{\varepsilon_0\to 0^+} K_{d,\sigma,\varepsilon_0} 
= \lim_{\varepsilon_0\to 0^+} \binom{x_0 + d}{d}\cdot x_0^{-d} \cdot\Bigl(\frac{2\log(y_0)}{\lamWo(\frac{y_0}{e\sigma^2})}\Bigr)^d
= \frac{2^d}{d!}
\eqspace.
\end{equation*}
This sharpens the result of \cite[Remark~4]{K2011} by a factor of approximately $\sqrt{2\pi d}$. 
Finally note that for $\sigma=1$ and $\varepsilon_0 \coloneqq 4\exp(-e^2)\approx 0.0025$ we have $y_0=e^2$ and $x_0 = 2 e^2$.
Hence 
we find
\begin{equation*}
K_{d,1,\varepsilon_0}\leq 16\cdot d^{2 e^2}\cdot \bigl(\sfrac{2}{e^2}\bigr)^d
\eqspace.
\end{equation*}
In this respect we like to mention that \cite{Zh2002} established the constant $4^d (6 d + 2)$ for $0<\varepsilon\leq \exp(-90 d^2 - 11 d - 3)$, 
again however, for a $\log^{d+1}(\sfrac{1}{\varepsilon})$-type bound on $X = [0,1]^d$. To compare this result of \cite{Zh2002} with our Theorem~\ref{thm:results:entropy-2} we use $[0,1]^d\subseteq \sfrac{1}{2} + \sfrac{\sqrt{d}}{2}\cdot B_2^d$ in combination with the later established Theorem~\ref{thm:results:decomposition} and \eqref{eq:results:volume} for $r = \sfrac{\sqrt{d}}{2}$ and $\sigma=1$ as well as Lemma~\ref{lem:entropy:range_const} for $C\coloneqq\sfrac{1}{\sqrt{360 e}}$ to obtain a constant not exceeding 
\begin{equation*}
(2\pi)^{-\sfrac{1}{2}}\cdot 16.84^d\cdot d^{-\sfrac{d}{2}-\sfrac{1}{2}}
\end{equation*}
for the range $0<\varepsilon\leq \varepsilon_0 \coloneqq 4\exp\bigl(- 3 d \sqrt{10 e} \log\bigl(3 d \sqrt{\sfrac{10}{e}}\bigr)\bigr)$ on $X=[0,1]^d$. This improves the result from \cite{Zh2002} in both, the $\varepsilon$ range and the constant for all $d\geq 1$.

For some applications, see e.g.\ \cite{VaZa2009,FaSt2018}, it is sufficient and more convenient to work with a weaker 
bound in $\varepsilon$ such as the one in \eqref{eq:intro:example_upper_bound}. 
For this reason, the following theorem establishes an upper bound of the form 
\eqref{eq:intro:example_upper_bound} with an explicit constant.

\begin{thm}\label{thm:results:poly}
For $d\geq 1$ and $\sigma>0$ we define 
\begin{equation*}
t_0 
\coloneqq \frac{2 (d+1)\cdot 4^{\frac{p}{d+1}}}{e p \cdot \lamWo(\frac{d+1}{p \sigma^2})}\exp\Bigl(\frac{1}{\lamWo(\frac{d+1}{p \sigma^2})}\Bigr)
\eqspace.
\end{equation*}
Then for all $d\geq 1$, $\sigma>0$, $p>0$, and all $0<\varepsilon\leq 1$ we have
\begin{equation*}
\coveringl[big]{\varepsilon}{\Id: H_\sigma(B_2^d)\to\ell_\infty(B_2^d)} 
\leq 
\binom{t_0 + d}{d} \cdot \frac{d+1}{e p}\cdot 4^{\frac{p}{d+1}}\cdot \varepsilon^{-p}
\eqspace.
\end{equation*}
\end{thm}

To better understand the constant appearing in 
Theorem~\ref{thm:results:poly}, we denote it by 
\begin{equation}\label{eq:results:poly_const-1}
K_{d,\sigma,p}
\coloneqq \binom{t_0 + d}{d} \cdot \frac{d+1}{e p }\cdot 4^{\frac{p}{d+1}}
\eqspace.
\end{equation}
For fixed $\sigma,p>0$, Lemma~\ref{lem:entropy:poly_const_d} then shows that $d\mapsto K_{d,\sigma,p}$ grows 
more slowly than any exponential function, i.e. for all $a>0$ we have $K_{d,\sigma,p}e^{-a d} \to 0$ for $d\to\infty$. To be more precise, Lemma~\ref{lem:entropy:poly_const_d} provides constants $c_{\sigma,p}>0$ and 
$C_{\sigma,p}>0$ independent of $d$ such that 
\begin{equation*}
K_{d,\sigma,p}
\leq C_{\sigma,p} \sqrt{d \log(d)}\cdot \exp\Bigl(c_{\sigma,p}\cdot d\cdot \frac{\log\log(d)}{\log(d)}\Bigr)
\eqspace, \qquad \qquad d\geq 1.
\end{equation*}
Moreover, if we restrict our considerations to $\sigma\coloneqq 1$ and also fix a $d\geq 1$ and a $0<p_0\leq \sfrac{1}{e}$,
then Lemma~\ref{lem:entropy:poly_const_p} shows that 
\begin{equation*}
K_{d,1,p}
\leq \sfrac{1}{2}\cdot C_0^d \cdot\sqrt{d}\cdot \frac{(\sfrac{1}{p})^{d+1}}{\log^d(\sfrac{1}{p})}
\eqspace, \qquad \qquad 0<p\leq p_0,
\end{equation*}
where the constant $C_0$ is given by 
\begin{equation}\label{eq:results:poly_const-2}
C_0 \coloneqq ep_0\cdot\log(\sfrac{1}{p_0}) + (2 + \sfrac{1}{e}) 2^{1+p_0}\exp\Bigl(\frac{1}{\lamWo(\sfrac{2}{p_0})}\Bigr)
\eqspace.
\end{equation}
In particular, $C_0$ only depends on $p_0$, and for $p_0 \coloneqq \sfrac{1}{e}$ we find $C_0 \approx 13.6481$. In addition,
$C_0$ converges to $4+\sfrac{2}{e}\approx 4.7358$ for $p_0\to 0$.
Finally, we note that the constant appearing in Theorem~\ref{thm:results:poly} can again be substantially improved if we restrict our consideration to a smaller range $0<\varepsilon\leq\varepsilon_0$. 

\subsection{Anisotropic Gaussian Kernels}\label{sec:anisotropic}

The goal of this subsection is to analyze how the constants in the log-covering number bounds 
depend on the kernel width $\sigma$ and the size of the input space $X$. To this end, 
our next theorem reduces the problem of bounding the log-covering numbers of 
$\Id:H_{\vecsigma}(X)\to\ell_\infty(X)$ of \emph{anisotropic} Gaussian RKHS 
to the estimation of the log-covering numbers of the embedding $\Id:H_1(B_2^d)\to\ell_\infty(B_2^d)$ of the \emph{isotropic} Gaussian RKHS with width $\sigma=1$.

\begin{thm}\label{thm:results:decomposition}
For all bounded subsets $X\subseteq\R^d$, all $\vecsigma=(\sigma_1,\ldots,\sigma_d)\in(0,\infty)^d$, and all 
 $0<\varepsilon\leq 1$ we have
\begin{equation*}
\coveringl[big]{\varepsilon}{\Id: H_\vecsigma(X)\to\ell_\infty(X)}
\leq \covering{1}{D_\vecsigma X}\cdot \coveringl[big]{\varepsilon}{\Id: H_1(B_2^d)\to\ell_\infty(B_2^d)}
\eqspace, 
\end{equation*}
where the covering numbers $\covering{1}{D_\vecsigma X}$ of $D_\vecsigma X\subseteq\R^d$ are with respect to the Euclidean norm.
\end{thm}

Before we proceed we like to remark that Theorem \ref{thm:results:decomposition} 
actually holds for general bounded and translation invariant kernels, see Section~\ref{sec:decomposition} for details.

Now, to illustrate the impact of Theorem \ref{thm:results:decomposition} we note that 
$X$ is assumed to be bounded, and hence there are an $x\in\R^d$ and an $r>0$ with $X\subseteq x + r B_2^d$. 
In the case of $\min_i\sigma_i\geq \sfrac{1}{r}$, Lemma~\ref{lem:decomposition:volume} 
then gives us
\begin{equation}\label{eq:results:volume}
\covering{1}{D_\vecsigma X}
\leq \covering{\sfrac{1}{r}}{D_\vecsigma B_2^d}
\leq \sigma_1\cdot\ldots\cdot\sigma_d \cdot(3r)^d
\eqspace.
\end{equation}
For the sake of completeness, we further mention that in the case of $\max_i\sigma_i\leq\sfrac{1}{r}$ we have $\covering{1}{D_\vecsigma X} = 1$.
Now, we can combine Theorem~\ref{thm:results:decomposition} with one of the theorems presented in Section~\ref{sec:isotropic}. For example, by combining Theorem~\ref{thm:results:decomposition} with Theorem~\ref{thm:results:entropy-1} and \eqref{eq:results:volume} we obtain
\begin{equation}\label{eq:results:final}
\coveringl[big]{\varepsilon}{\Id: H_\vecsigma(r B_2^d)\to\ell_\infty(r B_2^d)} 
\leq \binom{4 e + d}{d} \cdot (\sfrac{3 r}{e})^d\cdot \sigma_1\cdot\ldots\cdot\sigma_d\cdot \frac{\log^{d+1}(\sfrac{4}{\varepsilon})}{\log\log^{d}(\sfrac{4}{\varepsilon})}
\end{equation}
for all $0<\varepsilon\leq 1$, all $r>0$, and all $\vecsigma=(\sigma_1,\ldots,\sigma_d)\in [\sfrac{1}{r},\infty)^d$. 
Finally, we mention that in the case of $r\geq 1$ and an isotropic Gaussian kernel with width $\sigma\geq 1$ the constant in \eqref{eq:results:final}, that is
\begin{equation}\label{kd1-bound-2}
\tilde{K}_{d,\sigma,r}\coloneqq K_{d,1}\cdot (3 r \sigma)^d = \binom{4 e + d}{d} \cdot (\sfrac{3}{e})^d\cdot r^d\cdot \sigma^d
\eqspace,
\end{equation}
grows like $\sigma^d$ for $\sigma\to\infty$. In contrast, recall from \eqref{eq:results:sigma}
that the constant $K_{d,\sigma}$ obtained in Theorem~\ref{thm:results:entropy-1} grows like $\sigma^{2d}$.
Consequently, 
\eqref{eq:results:final} improves Theorem~\ref{thm:results:entropy-1} in the dependency on $\sigma$ by a factor of $2$ in the exponent. 
In this respect note that \cite{VaZa2009} obtained the same behavior in $\sigma$ but for a bound that 
does \emph{not} include the double logarithmic factor $\log\log^{d}(\sfrac{4}{\varepsilon})$ in \eqref{eq:results:final}.
Moreover, \cite[Theorem 6.27]{StCh2008} achieves the same behavior in $\sigma$ for a polynomial bound
of the form \eqref{eq:intro:example_upper_bound}. Of course, the latter two results can be recovered 
from \eqref{eq:results:final}, and in addition, the results in \cite{StCh2008,VaZa2009} do not take care of the explicit form of 
the constants.

\section{Proofs}\label{sec:proofs}


Before we present the proofs of our results, we briefly recall some basic facts about covering numbers. To this end, let $S,T:U\to V$ and $R:V\to W$ be some bounded operators between Banach spaces. Then the covering numbers satisfy
\begin{equation}\label{eq:covering:add_mult}
\covering{\varepsilon+\|T\|}{S+T} \leq \covering{\varepsilon}{S}
\qquad\text{and}\qquad
\covering{\varepsilon\|R\|}{RS}\leq \covering{\varepsilon}{S}
\end{equation}
for all $\varepsilon>0$. Furthermore, if $T$ has a finite rank, then the covering numbers satisfy the following standard bound
\begin{equation}\label{eq:covering:finite_rank}
\covering{\varepsilon}{T}\leq \biggl(1+\frac{2\|T\|}{\varepsilon}\biggr)^{\rank T}
\end{equation}
For the proofs of these properties and a comprehensive introduction to this topic we refer to \cite[Section~1.3]{CaSt1990}, where we note that in \cite{CaSt1990} the proofs are for entropy numbers but they easily transfer to covering numbers.


\subsection{Isotropic Gaussian Kernels}


Throughout this subsection 
the domain $X\coloneqq B_2^d\subseteq\R^d$ is fixed, and hence we simply write $I_\sigma$
for the embedding $\Id:H_\sigma(B_2^d)\to\ell_\infty(B_2^d)$.
Before we prove the 
results of Subsection \ref{sec:isotropic} we present several auxiliary lemmas. 
Our first result in this direction, which essentially repeats the key argument of 
\cite[Theorem~3]{K2011} on the input space $X=B_2^d$ instead of
$X=[0,1]^d$, provides a general estimate for the log-covering numbers of $I_\sigma$.

\begin{lem}\label{lem:entropy:kuehn}
For all $\sigma>0$, $\varepsilon>0$,
and all integers $N\geq 1$ we have
\begin{equation*}
\coveringl[bigg]{\varepsilon + \sqrt{\frac{(2\sigma^2)^N}{N!}}}{I_\sigma} 
\leq \binom{N-1+d}{d}\cdot \log\bigl(1+\sfrac{2}{\varepsilon}\bigr)
\eqspace.
\end{equation*}
\end{lem}

\begin{proof}
For fixed $\sigma>0$, $\varepsilon>0$, and $N\geq 1$ we define
\begin{equation*}
\varepsilon_0 
\coloneqq \sqrt{\frac{(2\sigma^2)^N}{N!}}
\eqspace.
\end{equation*} 
In order to repeat the argument of \cite[Theorem~3]{K2011},
we begin by recalling 
some notation: For every multi-index $k=(k_1,\ldots, k_d)\in\N_0^d$ we define the function $e_{k}:B_2^d\to\R$ by
\begin{equation*}
e_{k}(x)\coloneqq \sqrt{\frac{(2\sigma^2)^{|k|}}{k!}} x^k \exp\bigl(- \sigma^2\|x\|_{\ell_2^d}^2\bigr)
\end{equation*}
where we use $|k|\coloneqq k_1+\ldots+k_d$, $k!=k_1!\cdot\ldots\cdot k_d!$, and $x^k\coloneqq x_1^{k_1}\cdot\ldots\cdot x_d^{k_d}$ for $x=(x_1,\ldots,x_d)\in B_2^d$. Since $B_2^d$ has a non-empty interior the family of functions $(e_{k})_{k\in\N_0^d}$ forms an orthonormal basis (ONB) of $H_{\sigma}(B_2^d)$ according to \cite[Theorem~4.42]{StCh2008}. Using this ONB we now consider, for $N\geq 1$, 
the orthogonal projections $P_N,Q_N:H_{\sigma}(B_2^d)\to H_{\sigma}(B_2^d)$ 
onto $\overline{\spanning}\{e_{k}:\ |k|< N\}$ and $\overline{\spanning}\{e_{k}:\ |k|\geq N\}$, respectively. From the first equation on page~494 of \cite{K2011} we know
\begin{equation*}
\|I_\sigma\circ Q_N\| 
\leq\sup_{x\in B_2^d} \sqrt{\frac{(2\sigma^2\|x\|_{\ell_2^d}^2)^N}{N!}} 
= \sqrt{\frac{(2\sigma^2)^N}{N!}}
= \varepsilon_0
\eqspace.
\end{equation*}
As a consequence of \eqref{eq:covering:add_mult}, \eqref{eq:covering:finite_rank}, and $\|I_\sigma\circ P_N\|=1$ we get 
\begin{equation*}
\coveringl[big]{\varepsilon+\varepsilon_0}{I_\sigma}
= \coveringl[Big]{\varepsilon+\varepsilon_0}{I_\sigma\circ P_N + I_\sigma\circ Q_N}
\leq \coveringl[big]{\varepsilon}{I_\sigma\circ P_N}
\leq \rank(P_N)\log\bigl(1 + \sfrac{2}{\varepsilon}\bigr)
\eqspace.
\end{equation*}
Together with the formula 
\begin{equation*}
\rank(P_N)=\binom{N-1+d}{d}
\eqspace,
\end{equation*}
which was derived in \cite[Remark~4]{K2011},
we thus obtain the assertion.
\end{proof}

Our next goal is to find suitable values of $N\geq 1$ 
for the bound in Lemma~\ref{lem:entropy:kuehn}.
To this end,
recall that Lambert's $W$-function is 
increasing and satisfies the relations $W(x)>0$ for $x>0$, 
$\lamWo(x)e^{\lamWo(x)} = x$ for $x\geq -\sfrac{1}{e}$, and $\lamWo(y e^y) = y$ for $y\geq -1$.
In the following, we will often use these relations 
without referencing them. 

\begin{lem}\label{lem:entropy:ef}
For all $\sigma >0$, $x>0$, $y\geq -\sigma^2$,
\begin{align*}
\ef{\sigma}(x) 
\coloneqq 2\Bigl(\frac{2 e \sigma^2}{x}\Bigr)^{\sfrac{x}{2}}
\eqspace,\qquad\text{and}\qquad
\efii{\sigma}(y)
\coloneqq 2 e \sigma^2\exp\Bigl(\lamWo\Bigl(\frac{y}{e \sigma^2}\Bigr)\Bigr) 
\end{align*} 
the following statements are true:
\begin{enumerate}
\item\label{it:entropy:ef:i} The function $\ef{\sigma}:(0,\infty)\to (0,\infty)$ is 
decreasing on $(2 \sigma^2,\infty)$ and $\lim_{x\to \infty}\ef{\sigma}(x)=0$.

\item\label{it:entropy:ef:ii} The function $\efii{\sigma}:[-\sigma^2, \infty)\to [2\sigma^2,\infty)$ is
increasing and we have 
\begin{equation}\label{eq:entropy:efii}
\efii{\sigma}(y)
= \frac{2y}{\lamWo\bigl(\frac{y}{e \sigma^2}\bigr)}
\eqspace.
\end{equation}

\item\label{it:entropy:ef:iii} The function $\ef{\sigma}:[2 \sigma^2,\infty)\to (0,2\exp(\sigma^2)]$ is bijective with inverse $\efi{\sigma}$ given by
\begin{equation*}
\efi{\sigma}(\varepsilon)
= \efii{\sigma}\circ\log(\sfrac{2}{\varepsilon}) 
\eqspace.
\end{equation*}
\end{enumerate}
\end{lem}

\begin{proof}
\ada{\ref{it:entropy:ef:i}} 
Some tedious calculations show that the derivative of $\ef{\sigma}$
is given by
\begin{equation*}
\ef{\sigma}'(x) = \frac{\ef{\sigma}(x)}{2} \log\Bigl(\frac{2 \sigma^2}{x}\Bigr)
\eqspace, \qquad \qquad x>0.
\end{equation*}
From this identity the first assertion immediately follows. The second assertion is obvious.

\ada{\ref{it:entropy:ef:ii}} The monotonicity of $\efii{\sigma}$ is a consequence of the monotonicity of $\lamWo$ and the definition of the function $\efii{\sigma}$. Moreover, 
\eqref{eq:entropy:efii} follows from the identity
 $\lamWo(x)\exp(\lamWo(x)) = x$. 

\ada{\ref{it:entropy:ef:iii}} By part~\ref{it:entropy:ef:i} we already know that $\ef{\sigma}:[2 \sigma^2,\infty)\to (0,2\exp(\sigma^2)]$ is bijective. 
To verify the formula for $\efi{\sigma}$, we fix some 
 $0<\varepsilon\leq 2 \exp(\sigma^2)$ and write 
$y\coloneqq\log(\sfrac{2}{\varepsilon})$. This immediately gives $y\geq -\sigma^2$
 and by the definition of 
 $\efii{\sigma}$ we find
\begin{equation*}
\ef{\sigma}\circ\efii{\sigma}(y) 
= 2\Bigl(\frac{2 e \sigma^2}{\efii{\sigma}(y)}\Bigr)^{\sfrac{\efii{\sigma}(y)}{2}}
= 2\exp\Bigl(-\lamWo\Bigl(\frac{y}{e \sigma^2}\Bigr)\cdot \frac{\efii{\sigma}(y)}{2}\Bigr)
= 2 e^{-y} = \varepsilon
\eqspace,
\end{equation*}
i.e.\ we have shown the assertion.
\end{proof}

In the next lemma we choose a suitable parameter $N\geq 1$ for the bound in Lemma~\ref{lem:entropy:kuehn} with the help of the functions introduced in Lemma~\ref{lem:entropy:ef}.

\begin{lem}\label{lem:entropy:general}
For all $\sigma>0$ and all $0<\varepsilon\leq 1$, we have
\begin{equation}\label{eq:entropy:general}
\coveringl[big]{\varepsilon}{I_\sigma} 
\leq \binom{(\efii{\sigma}\circ\log)(\sfrac{4}{\varepsilon}) + d}{d}\cdot \log\bigl(\sfrac{4}{\varepsilon}\bigr)
\eqspace.
\end{equation}
\end{lem}

\begin{proof}
For a fixed $0<\varepsilon\leq 1$ we write $y\coloneqq\log(\sfrac{4}{\varepsilon})$ and $x\coloneqq\efii{\sigma}(y)$. Since $y>1$ we have $x>2\sigma^2$, and hence there is a unique integer $N\geq 1$ with $N-1<x\leq N$. Using Lemma~\ref{lem:entropy:kuehn} with $\sfrac{2\varepsilon}{3}$ instead of $\varepsilon$, the monotonicity of $t\mapsto\binom{t+d}{d}$, and $1\leq \sfrac{1}{\varepsilon}$ we find
\begin{equation*}
\coveringl[bigg]{\frac{2\varepsilon}{3} + \sqrt{\frac{(2\sigma^2)^N}{N!}}}{I_\sigma} 
\leq \binom{N-1+d}{d}\cdot \log\bigl(1+\sfrac{3}{\varepsilon}\bigr)
\leq \binom{x+d}{d}\cdot \log\bigl(\sfrac{4}{\varepsilon}\bigr)
\eqspace.
\end{equation*}
Consequently, it remains to show that $\sqrt{\sfrac{(2\sigma^2)^N}{N!}} \leq \sfrac{\varepsilon}{3}$ holds true. 
To this end, we use Stirling's formula $N!\geq \sqrt{2\pi N} (\sfrac{N}{e})^N$ to get
\begin{equation*}
\sqrt{\frac{(2 \sigma^2)^N}{N!}}
\leq \frac{1}{(2\pi N)^{\sfrac{1}{4}}}\cdot\Bigl(\frac{2 e \sigma^2}{N}\Bigr)^{\sfrac{N}{2}}
\leq \frac{\ef{\sigma}(N)}{2(2\pi)^{\sfrac{1}{4}}} 
\eqspace.
\end{equation*}
Moreover, the already observed $x>2\sigma^2$ together with parts~\ref{it:entropy:ef:i} and \ref{it:entropy:ef:iii} of Lemma~\ref{lem:entropy:ef} yields
\begin{equation*}
\ef{\sigma}(N) \leq \ef{\sigma}(x) 
= \ef{\sigma}\bigl( \efii{\sigma}(y) \bigr)
= \ef{\sigma}\bigl( \efii{\sigma}\circ\log(\sfrac{4}{\varepsilon}) \bigr)
= \sfrac{\varepsilon}{2}
\eqspace.
\end{equation*}
Combining both estimates and $(2\pi)^{-\sfrac{1}{4}}\leq \sfrac{4}{3}$ we get the assertion.
\end{proof}

Note that by an easy adaption of the above proof we can replace the $4$ in $y=\log(\sfrac{4}{\varepsilon})$ by $\gamma=\sfrac{7}{2}$ if we choose $\sfrac{4\varepsilon}{5}$ instead of $\sfrac{2\varepsilon}{3}$ and use the bound $(2\pi)^{-\sfrac{1}{4}}\approx 0.6316 \leq \sfrac{7}{10}$. Moreover, some tedious calculations show that the argument still works for
\begin{equation*}
\gamma 
\coloneqq \frac{3(2\pi)^{\sfrac{1}{4}} + 1 + \sqrt{9 (2\pi)^{\sfrac{1}{2}} + 2 (2\pi)^{\sfrac{1}{4}} + 1}}{2 (2\pi)^{\sfrac{1}{4}}}
\approx 3.4485
\eqspace.
\end{equation*}
Since these improvements have little impact we stick to $\gamma=4$ for convenience.
The following lemma demonstrates the general technique we use to bound the right hand side of \eqref{eq:entropy:general}.

\begin{lem}\label{lem:entropy:binom_bound}
Let $d\geq 1$, $\sigma>0$, $0<\varepsilon_0\leq 1$, and $t_0>0$. If $f:(0,\varepsilon_0]\to (0,\infty)$ is a function with $f(\varepsilon)\geq t_0$ and
\begin{equation*}
(\efii{\sigma}\circ\log)(\sfrac{4}{\varepsilon}) \leq f(\varepsilon)
\end{equation*}
for all $0<\varepsilon\leq\varepsilon_0$, then we have
\begin{equation*}
\binom{(\efii{\sigma}\circ\log)(\sfrac{4}{\varepsilon})+d}{d} \leq \binom{t_0 + d}{d}\cdot \Bigl(\frac{f(\varepsilon)}{t_0}\Bigr)^d
\eqspace.
\end{equation*}
\end{lem}

\begin{proof}
In order to prove this statement we use the auxiliary function $G_d(t):(0,\infty)\to(0,\infty)$ defined by
$G_d(t)
\coloneqq\binom{t+d}{d}\cdot t^{-d} 
= \frac{1}{d!}\prod_{i=1}^d\bigl(1 + \frac{i}{t}\bigr)$.
Since $t \mapsto\binom{t+d}{d}$ is increasing and $G_d$ is decreasing we get
\begin{equation*}
\binom{(\efii{\sigma}\circ\log)(\sfrac{4}{\varepsilon}) + d}{d}
\leq \binom{f(\varepsilon)+d}{d}
= G_d(f(\varepsilon))\cdot f^d(\varepsilon)
\leq G_d(t_0)\cdot f^d(\varepsilon)
\end{equation*}
for all $0<\varepsilon\leq\varepsilon_0$, which gives the assertion.
\end{proof}

As final preparation we need the following simple lemma.

\begin{lem}\label{lem:entropy:lambW}
For $\sigma>0$ consider $t^* \coloneqq \sigma^{-2} \exp(\sigma^{-2})$ and $\qs:(0,\infty)\to \R$ defined by 
\begin{equation*}
\qs(t) \coloneqq \frac{1+ \log(\sigma^2)+ \log(t)} {\lamWo(t)}
\eqspace.
\end{equation*}
Then $\qs$ is increasing on $(0,t^\ast]$ and decreasing on $[t^\ast,\infty)$. Moreover, $\qs$ has a unique global maximum at $t^\ast$ with $\qs(t^\ast) = 1+\sigma^2$ and we have $\lim_{t\to \infty}\qs(t) = 1$.
\end{lem}

\begin{proof} 
A simple but tedious calculation shows
\begin{equation*}
\qs'(t) = \frac{\lamWo(t) - \log (t\sigma^2)}{t \cdot \lamWo(t) \cdot (1+ \lamWo(t))}
\eqspace.
\end{equation*}
Since the denominator is positive for all $t>0$ we can focus on the numerator in order to investigate the monotonicity properties of $\qs$. Consequently, $\qs$ is decreasing, if and only if $W(t) < \log(t\sigma^2)$ and this is equivalent to
\begin{equation*}
t 
= W(t) e^{W(t)} 
< \log(t\sigma^2)\cdot \exp\circ\log(t\sigma^2)
= t\sigma^2 \log(t\sigma^2)
\eqspace.
\end{equation*}
Rearranging this inequality for $t$ shows that $\qs$ is decreasing on $[t^\ast,\infty)$. Analogously, we get that $\qs$ is increasing on $(0,t^\ast]$ and that $\qs$ has a unique global maximum at $t^\ast$.
Since $\lamWo(t^\ast) = \sigma^{-2}$ and $\log(t^\ast) = -\log(\sigma^2) + \sigma^{-2}$ we find $\qs(t^\ast) = 1+\sigma^2$.
Finally, for $t\geq 1$ we write $s\coloneqq t e^t$ and from 
\begin{equation*}
\qs(s) 
= \frac{1+ \log(\sigma^2)+ \log(t e^t)} {\lamWo(t e^t)}
= \frac{1+ \log(\sigma^2)+ \log(t) + t} {t}
\eqspace,
\end{equation*}
the assertion $\lim_{t\to\infty}\qs(t)=1$ easily follows.
\end{proof}

\begin{proof}[Proof of Theorem~\ref{thm:results:entropy-1}]
Let us define $\varepsilon_0 \coloneqq 1$ and $y_0 \coloneqq \log(\sfrac 4 {\varepsilon_0})$. 
For $0<\varepsilon\leq \varepsilon_0$ we further write 
$y\coloneqq\log(\sfrac{4}{\varepsilon})\geq y_0>1$. 
An application of 
Lemma~\ref{lem:entropy:lambW} then yields
\begin{equation*}
(\efii{\sigma}\circ \log)(\sfrac{4}{\varepsilon})
=\frac{2y}{\lamWo(\frac{y}{e\sigma^2})}
= \frac{2y}{\log(y)}\cdot \frac{1+\log(\sigma^2) + \log(\frac{y}{e\sigma^2})}{\lamWo(\frac{y}{e\sigma^2})}
\leq 2\,(1+\sigma^2)\cdot\frac{y}{\log(y)}
\eqqcolon f(\varepsilon)
\end{equation*}
for all $0<\varepsilon\leq\varepsilon_0$.
Now, the derivative of 
$\beta:(1,\infty)\to (0,\infty)$ defined by 
$\beta(t) \coloneqq \frac t{\log (t)}$ is 
\begin{equation}\label{eq:entropy:entropy:temp}
\beta'(t) 
= \frac{\log(t)-1}{\log^2(t)}
\eqspace,
\end{equation}
and consequently
it is easy to check that, for $t^\ast \coloneqq e$, the function $\beta$ is decreasing on $(1,t^\ast]$, increasing on $[t^\ast,\infty)$, and has a unique global minimum at $t^*$ with $\beta(t^\ast) = e$. As a result, we get $f(\varepsilon)\geq 2 e (1+\sigma^2) \eqqcolon t_0$. 
Finally, combining Lemma~\ref{lem:entropy:general} and Lemma~\ref{lem:entropy:binom_bound} gives the assertion.
\end{proof}

\begin{proof}[Proof of Theorem~\ref{thm:results:entropy-2}]
For a fixed $0<\varepsilon_0 \leq 4\exp(-e^{1 + \sigma^{-2}})$
we recall the definitions of $y_0 \coloneqq \log(\sfrac{4}{\varepsilon_0})$
and $x_0 \coloneqq \efii{\sigma}(y_0)$.
Moreover, for 
$0<\varepsilon\leq \varepsilon_0$ we write $y\coloneqq \log(\sfrac{4}{\varepsilon}) \geq y_0$. 
Note that
the restriction on $\varepsilon_0$ ensures $y_0 \geq \exp(1+\sigma^{-2})$ and hence $\frac{y_0}{e\sigma^2} \geq \sigma^{-2}\exp(\sigma^{-2})$. As a consequence, the function $y\mapsto\sfrac{\log(y)}{\lamWo(\frac{y}{e\sigma^2})}$ is decreasing on $[y_0,\infty)$ according to Lemma~\ref{lem:entropy:lambW} and we get
\begin{equation*}
(\efii{\sigma}\circ\log)(\sfrac{4}{\varepsilon})
=\frac{2\log(y)}{\lamWo(\frac{y}{e \sigma^2})}\cdot\frac{y}{\log(y)}
\leq
\frac{2\log(y_0)}{\lamWo(\frac{y_0}{e\sigma^2})}
\cdot 
\frac{y}{\log (y)} 
\eqqcolon f(\varepsilon)
\eqspace.
\end{equation*}
Now, from \eqref{eq:entropy:entropy:temp} 
we know that the function $\beta(t)=\frac{t}{\log(t)}$ is increasing on $[e,\infty)$, and hence
\begin{equation*}
f(\varepsilon)
\geq \frac{2\log(y_0)}{\lamWo(\frac{y_0}{e\sigma^2})}
\cdot 
\frac{y_0}{\log (y_0)} 
=\frac{2y_0}{\lamWo(\frac{y_0}{e\sigma^2})} 
= x_0 
\eqqcolon t_0
\end{equation*}
for all $0<\varepsilon\leq\varepsilon_0$. 
Finally, combining Lemma~\ref{lem:entropy:general} and Lemma~\ref{lem:entropy:binom_bound} gives the assertion.
\end{proof}

\begin{proof}[Proof of Theorem~\ref{thm:results:poly}]
For $0<\varepsilon\leq 1$ we again write $y\coloneqq\log(\sfrac{4}{\varepsilon})\geq \log(4)$.
In order to give a polynomial upper bound for $\coveringl{\varepsilon}{I_\sigma}$ we use Lemma~\ref{lem:entropy:general} and estimate the two factors, 
$\binom{\efii{\sigma}(y) + d}{d}$
and $\log(\sfrac{4}{\varepsilon})$,
appearing in
\eqref{eq:entropy:general} separately by a polynomial bound. 
To bound the first factor we fix a $q_1>0$ and define the function 
\begin{equation*}
g_1(t)\coloneqq 2 \frac{t e^{-q_1 t}}{\lamWo(\frac{t}{e\sigma^2})}
\eqspace, \qquad\qquad t>0.
\end{equation*}
Using $e^{-q_1 y} = (\sfrac{4}{\varepsilon})^{-q_1}$ we then get
\begin{equation*}
(\efii{\sigma}\circ\log)(\sfrac{4}{\varepsilon})
=\frac{2y}{\lamWo(\frac{y}{e\sigma^2})} \Bigl(\frac{4}{\varepsilon}\Bigr)^{-q_1} \cdot \Bigl(\frac{4}{\varepsilon}\Bigr)^{q_1} 
\leq \Bigl(\frac{4}{\varepsilon}\Bigr)^{q_1} \sup_{t>0} g_1(t)
\eqqcolon f(\varepsilon)
\end{equation*}
and $f(\varepsilon)\geq 4^{q_1}\cdot \sup_{t>0} g_1(t) \eqqcolon t_0$. 
A simple but tedious calculation shows
\begin{equation*}
g_1'(t) = \frac{g_1(t)}{1+\lamWo(\frac{t}{e\sigma^2})}\biggl(\sigma^{-2}\exp\Bigl(-\Bigl(1 + \lamWo\Bigl(\frac{t}{e\sigma^2}\Bigr)\Bigr)\Bigr) - q_1\Bigl(1 + \lamWo\Bigl(\frac{t}{e\sigma^2}\Bigr)\Bigr)\biggr)
\eqspace.
\end{equation*}
If we define
\begin{equation*}
t^\ast\coloneqq\frac{1}{q_1}\Bigl(1 - \frac{1}{\lamWo(\frac{1}{q_1\sigma^2})}\Bigr)
\end{equation*}
then another tedious calculation shows that $g_1$ is increasing on $(0,t^\ast]$, decreasing on $[t^\ast,\infty)$, and has a unique global maximum at $t^\ast$. In order to evaluate the maximum $g_1(t^\ast)$ we first give another representation of $t^\ast$ using $\sfrac{t}{\lamWo(t)} = \exp(\lamWo(t))$ for $t=\frac{1}{q_1\sigma^2}$
\begin{equation*}
t^\ast 
= \Bigl(\lamWo\Bigl(\frac{1}{q_1\sigma^2}\Bigr) - 1\Bigr)\cdot\frac{ \sfrac{1}{q_1}}{\lamWo(\frac{1}{q_1\sigma^2})}
= \sigma^2\cdot \Bigl(\lamWo\Bigl(\frac{1}{q_1\sigma^2}\Bigr) - 1\Bigr) \cdot\exp\circ\lamWo\Bigl(\frac{1}{q_1\sigma^2}\Bigr)
\eqspace.
\end{equation*}
Using this representation together with 
$\lamWo(xe^x)=x$
for $x=\lamWo(\frac{1}{q_1\sigma^2}) - 1$ gives us 
\begin{equation*}
\frac{t^\ast}{\lamWo(\frac{t^\ast}{e\sigma^2}) }
= \frac{t^\ast}{\lamWo (\frac{1}{q_1\sigma^2} ) - 1}
= \frac{1}{q_1\cdot\lamWo(\frac{1}{q_1\sigma^2})}
\eqspace.
\end{equation*}
Using this identity we directly get
\begin{equation*}
t_0 
= 4^{q_1} g_1(t^\ast) 
= 2\cdot 4^{q_1}\cdot \frac{e^{-q_1 t^\ast}}{q_1\cdot\lamWo(\frac{1}{q_1\sigma^2})}
= \frac{2\cdot 4^{q_1}}{e q_1\cdot \lamWo(\frac{1}{q_1\sigma^2})}\exp\Bigl( \sfrac{1}{\lamWo\Bigl(\frac{1}{q_1\sigma^2}\Bigr)}\Bigr)
\end{equation*}
and Lemma~\ref{lem:entropy:binom_bound} gives us
\begin{equation}\label{eq:entropy:poly:temp-1}
\binom{(\efii{\sigma}\circ\log)(\sfrac{4}{\varepsilon})+d}{d} 
\leq \binom{t_0 + d}{d}\cdot \Bigl(\frac{f(\varepsilon)}{t_0}\Bigr)^d
= \binom{t_0 + d}{d}\cdot 4^{-q_1 d}\cdot (\sfrac{4}{\varepsilon})^{q_1 d}
\eqspace.
\end{equation}

Now, we estimate the second factor $y=\log(\sfrac{4}{\varepsilon})$ by a polynomial bound of order $q_2>0$. To this end, we define the function $g_2(t)\coloneqq t e^{-q_2 t}$, for $t>0$, and estimate
\begin{equation*}
y 
= (\sfrac{4}{\varepsilon})^{q_2}\cdot y\cdot (\sfrac{4}{\varepsilon})^{-q_2} 
\leq (\sfrac{4}{\varepsilon})^{q_2} \cdot\sup_{t>0} g_2(t)
\eqspace.
\end{equation*}
An easy calculation shows that the derivative of $g_2$ is given by 
$g_2'(t) = g_2(t) \cdot(\sfrac{1}{t} - q_2)$
and consequently $g_2$ has a global maximum at $t^\ast \coloneqq \sfrac{1}{q_2}$ with $g_2(t^\ast) = \frac{1}{e q_2}$. Therefore, we get 
\begin{equation}\label{eq:entropy:poly:temp-2}
y 
\leq \frac{(\sfrac{4}{\varepsilon})^{q_2}}{e q_2}
\eqspace.
\end{equation}
Finally, combining Lemma~\ref{lem:entropy:general} with \eqref{eq:entropy:poly:temp-1} and \eqref{eq:entropy:poly:temp-2} yields
\begin{equation*}
\coveringl{\varepsilon}{I_\sigma}
\leq \binom{t_0 + d}{d}\cdot \frac{1}{e q_2\cdot 4^{q_1 d}}\cdot (\sfrac{4}{\varepsilon})^{q_1 d + q_2} 
\eqspace,
\end{equation*}
and for $q_1=q_2=\frac{p}{d+1}$ we get the assertion.
\end{proof}


\subsection{Auxiliary Results}


In this section we collect additional results that are helpful to understand the quantities appearing in the bounds presented in Section~\ref{sec:isotropic}.

\begin{lem}\label{lem:entropy:binom_behavior}
For an integer $d\geq 1$ and a real number $t>0$ the (generalized) binomial coefficient from \eqref{eq:results:binom} satisfies
\begin{equation*}
\binom{t+d}{d} 
= \frac{\Gamma(t+d+1)}{\Gamma(t+1)\Gamma(d+1)}
\eqspace.
\end{equation*}
Moreover, for a fixed real number $t>0$ the sequence
\begin{equation*}
a_d\coloneqq\binom{t + d}{d}\cdot d^{-t}
\eqspace, \qquad \qquad d\geq 1,
\end{equation*}
is decreasing and converges to $\sfrac{1}{\Gamma(t+1)}$.
\end{lem}

\begin{proof}
First note that $\Gamma(d+1)=d!$ and an $d$-times application of $\Gamma(t+1) = t\cdot \Gamma(t)$ gives us
\begin{equation*}
\binom{t+d}{d}
= \frac{1}{d!}\prod_{i=1}^d (t + i)
= \frac{\Gamma(t+1)\prod_{i=1}^d (t + i)}{\Gamma(d+1)\Gamma(t+1)} 
= \frac{\Gamma(d+t+1)}{\Gamma(d+1)\Gamma(t+1)}
\eqspace.
\end{equation*}
Using $\frac{\Gamma(d+t+1)}{\Gamma(d+1)d^t} \to 1$ for $d\to\infty$, which is a well-known property of the Gamma function,
we get
\begin{equation*}
a_d
= \frac{1}{\Gamma(t+1)}\cdot\frac{\Gamma(d+t+1)}{\Gamma(d+1)d^t}\to \frac{1}{\Gamma(t+1)}
\end{equation*}
for $d\to\infty$ and it remains to show the monotonicity. Using $\Gamma(t+1)=t\cdot \Gamma(t)$ twice we get
\begin{equation*}
a_{d+1}
= \frac{\Gamma(d+t+2)}{\Gamma(d+2)\Gamma(t+1)} (d+1)^{-t}
= a_d \cdot\Bigl(\frac{d}{d+1}\Bigr)^t\cdot \frac{d+t+1}{d+1}
\eqspace.
\end{equation*}
Consequently, $(a_d)_{d\geq 1}$ is decreasing if and only if $\frac{d+t+1}{d+1} < \bigl(\frac{d+1}{d}\bigr)^t$ is satisfied for all $d\geq 1$ and $t>0$. 
In order to prove this we fix some $d\geq 1$ and show that
\begin{equation*}
f_d(t)\coloneqq \Bigl(1+\frac{1}{d}\Bigr)^t - \Bigl(1+\frac{t}{d+1}\Bigr) > 0
\end{equation*}
is satisfied for all $t>0$. To this end, we calculate the first and second derivative
\begin{align*}
f_d'(t) &= \Bigl(1+\frac{1}{d}\Bigr)^t \log\Bigl(1+\frac{1}{d}\Bigr) - \frac{1}{d+1}\\
f_d''(t)&= \Bigl(1+\frac{1}{d}\Bigr)^t \log^2\Bigl(1+\frac{1}{d}\Bigr)
\eqspace.
\end{align*}
Using $\log(1+x)\geq \frac{x}{1+x}$, which holds for all $x>-1$, for $x=\sfrac{1}{d}$, we get 
\begin{equation*}
f_d'(0) 
= \log\Bigl(1+\frac{1}{d}\Bigr) - \frac{1}{d+1} 
\geq \frac{\sfrac{1}{d}}{1+\sfrac{1}{d}} - \frac{1}{d+1} 
= 0
\eqspace.
\end{equation*}
Together with $f_d''(t) > 0$ we get $f'_d(t) > 0$ for all $t > 0$. Finally, $f_d(0) = 0$ and $f_d'(t) > 0$ gives $f_d(t) > 0$ for all $t>0$ and hence the assertion is proven.
\end{proof}

\begin{lem}\label{lem:entropy:range_const}
For all $C>0$, $d\geq 2 C e^2$, and $\varepsilon_0 \coloneqq 4 \exp\bigl(-\frac{d}{2 C}\log\bigl(\frac{d}{2 e C}\bigr)\bigr)$ the condition $\varepsilon_0 \leq 4\exp(-e^{1+\sigma^{-2}})$ for $\sigma=1$ in Theorem~\ref{thm:results:entropy-2} is satisfied and the quantity $K_{d,1,\varepsilon_0}$ defined in Theorem~\ref{thm:results:entropy-2} satisfies
\begin{equation*}
K_{d,1,\varepsilon_0}\leq (2\pi)^{-\sfrac{1}{2}}\cdot (4 e)^d (1+C)^d\cdot d^{-(d+\sfrac{1}{2})}
\eqspace.
\end{equation*}
Moreover, for $\frac{1}{2 e^2} \geq C\geq \frac{1}{\sqrt{360 e}}$ we have
\begin{equation*}
\exp\bigl(-90 d^2 - 11d -3\bigr)
\leq \varepsilon_0 
\eqspace.
\end{equation*}
\end{lem}

\begin{proof}
Let us recall the definition of $y_0\coloneqq\log(\sfrac{4}{\varepsilon_0})$ and $x_0\coloneqq\frac{2y_0}{\lamWo(\sfrac{y_0}{e})} = \efii{1}(y_0)$ in Theorem~\ref{thm:results:entropy-2} where $\efii{1}$ is defined in Lemma~\ref{lem:entropy:ef}. 
Using the function $\ef{1}$ defined in Lemma~\ref{lem:entropy:ef} we can write
\begin{equation*}
\varepsilon_0 
= 4 \Bigl(\frac{2 e C}{d}\Bigr)^{\frac{d}{2C}} = 2\cdot\ef{1}(\sfrac{d}{C})
\eqspace.
\end{equation*}
Since $\sfrac{d}{C}\geq 2 e^2 \geq 2$ part~\ref{it:entropy:ef:iii} of Lemma~\ref{lem:entropy:ef}, which states $\efi{1} = \efii{1}\circ\log(\sfrac{2}{\cdot})$, is applicable and hence
\begin{equation*}
x_0 
= \efii{1}\circ\log(\sfrac{4}{\varepsilon_0}) 
= \efii{1}\circ\log\Bigl(\frac{2}{\ef{1}(\sfrac{d}{C})}\Bigr) 
= \sfrac{d}{C}
\eqspace.
\end{equation*}
Next, we prove the inequality $\varepsilon_0 \leq u\coloneqq 4\exp(-e^2)$. To this end, note that we have $\efii{1}\circ\log(\sfrac{4}{u}) = \efii{1}(e^2) = 2 e \exp(\lamWo(e)) = 2 e^2$ since $\lamWo(e)=1$. Our assumption $d\geq 2 C e^2$ implies
\begin{equation*}
\efii{1}\circ\log(\sfrac{4}{\varepsilon_0})
=\sfrac{d}{C}
\geq
2 e^2 
= \efii{1}\circ\log(\sfrac{4}{u})
\end{equation*}
and since $\efii{1}$ is increasing according to part~\ref{it:entropy:ef:ii} of Lemma~\ref{lem:entropy:ef} we get $\varepsilon_0 \leq u = 4 \exp(-e^2)$.
Now, we prove the bound on $K_{d,1,\varepsilon_0}$. To this end, we rewrite $K_{d,1,\varepsilon_0}$ using the representation of the binomial coefficient from \eqref{eq:results:binom}
\begin{align*}
K_{d,1,\varepsilon_0}
&= \binom{x_0 + d}{d} x_0^{-d} \cdot\Bigl(\frac{x_0\log(y_0)}{y_0}\Bigr)^d\\
&= \frac{1}{d!}\prod_{i=1}^d (1 + \sfrac{i}{x_0})\cdot\Bigl(\frac{2\log(y_0)}{\lamWo(\sfrac{y_0}{e})}\Bigr)^d
\eqspace.
\end{align*}
If we bound the first factor by using $\sfrac{i}{x_0} \leq \sfrac{d}{x_0} = C$ and if we bound the second factor by using $\log(y_0)\leq 2 \lamWo(\sfrac{y_0}{e})$ from Lemma~\ref{lem:entropy:lambW} then we get
\begin{equation*}
K_{d,1,\varepsilon_0} 
\leq \frac{4^d (1+C)^d}{d!}
\eqspace.
\end{equation*}
Together with Stirling's formula $d! \geq \sqrt{2\pi d}\cdot (\sfrac{d}{e})^d$ this gives the desired bound. Finally, note that $\log(t)\leq t$ and $C\geq \sfrac{1}{\sqrt{360 e}}$ yields 
\begin{equation*}
\varepsilon_0 
\geq \exp\Bigl(-\frac{d^2}{4 e C^2}\Bigr)
\geq \exp\bigl(-90 d^2\bigr)
\eqspace,
\end{equation*}
which proves the lower bound on $\varepsilon_0$.
\end{proof}

\begin{lem}\label{lem:entropy:poly_const_d}
For $\sigma, p>0$ there are constants $c_{\sigma, p}, C_{\sigma, p}>0$ such that $K_{d,\sigma,p}$ defined in \eqref{eq:results:poly_const-1} satisfies, for all $d\geq 1$
\begin{equation*}
K_{d,\sigma,p}
\leq C_{\sigma, p}\cdot \sqrt{d \log(d)}\cdot \exp\Bigl(c_{\sigma, p}\cdot d\cdot \frac{\log\log(d)}{\log(d)}\Bigr)
\eqspace.
\end{equation*}
\end{lem}

\begin{proof}
For this proof we use the usual notation $a_n\preccurlyeq b_n$ for two sequences $(a_n)_{n\geq 1}, (b_n)_{n\geq 1}$ iff there is a constant $c>0$ with $a_n\leq c b_n$ for all $n\geq 1$. Moreover, we write $a_n\asymp b_n$ 
iff both $a_n\preccurlyeq b_n$ and $a_n\succcurlyeq b_n$ hold.
Using $\lamWo(t)\sim\log(t)$ we get, for $d\to\infty$,
\begin{equation*}
t_0 
= \frac{2 (d+1)\cdot 4^{\frac{p}{d+1}}}{e p \cdot \lamWo(\frac{d+1}{p \sigma^2})}\exp\Bigl(\frac{1}{\lamWo(\frac{d+1}{p \sigma^2})}\Bigr)
\asymp \frac{d}{\lamWo(\frac{d+1}{p\sigma^2})}
\asymp \frac{d}{\log(\frac{d+1}{p\sigma^2})}
\asymp \frac{d}{\log(d)}
\eqspace.
\end{equation*}
Since $t_0\to\infty$ for $d\to\infty$, Lemma~\ref{lem:entropy:binom_behavior} and Stirling's formula $\Gamma(t+1)\sim \sqrt{2\pi t}\,(\sfrac{t}{e})^t$ yield
\begin{equation*}
\binom{t_0 + d}{d} 
= \frac{\Gamma(t_0 + d + 1)}{\Gamma(t_0 + 1)\Gamma(d + 1)}
\asymp\Bigl(\frac{1}{t_0}+\frac{1}{d}\Bigr)^{\sfrac{1}{2}}\Bigl(1+\frac{t_0}{d}\Bigr)^{d}\Bigl(1+\frac{d}{t_0}\Bigr)^{t_0}
\end{equation*}
for $d\to\infty$. Using the inequality $1+t\leq e^t$, which holds for all $t\in\R$, for $t=\sfrac{t_0}{d}$ we get
\begin{equation*}
\binom{t_0 + d}{d}
\preccurlyeq \sqrt{\frac{\log(d)}{d}}\cdot e^{t_0}\cdot\bigl(1+\sfrac{d}{t_0}\bigr)^{t_0}
\end{equation*}
for $d\to \infty$.
Consequently, we find a constant $C_{\sigma, p}>0$ with
\begin{equation*}
K_{d,\sigma,p}
= \binom{t_0 + d}{d} \cdot \frac{d+1}{e p }\cdot 4^{\frac{p}{d+1}}
\leq C_{\sigma, p}\cdot \sqrt{d \log(d)}\cdot\exp\Bigl(t_0\cdot\log\Bigl(e + e\frac{d}{t_0}\Bigr)\Bigr)
\eqspace.
\end{equation*}
Since, for $d\to \infty$, the exponent behaves like
\begin{equation*}
t_0\cdot\log\Bigl(e + e\frac{d}{t_0}\Bigr) 
\asymp \frac{d}{\log(d)}\cdot \log\Bigl(e + e\log(d)\Bigr) 
\asymp d\cdot\frac{\log\log(d)}{\log(d)}
\end{equation*}
there is a constant $c_{\sigma, p}>0$ independent of $d$ with the desired property.
\end{proof}

\begin{lem}\label{lem:entropy:poly_const_p}
For $0<p_0\leq \sfrac{1}{e}$ the quantities $K_{d,1,p}$ and $C_0$ defined in \eqref{eq:results:poly_const-1} and \eqref{eq:results:poly_const-2}, respectively, satisfy
\begin{equation*}
K_{d,1,p}
\leq \sfrac{1}{2}\cdot C_0^d \cdot\sqrt{d}\cdot \frac{(\sfrac{1}{p})^{d+1}}{\log^d(\sfrac{1}{p})}
\eqspace,\qquad\qquad 0<p\leq p_0,\ d\geq 1.
\end{equation*}
\end{lem}

\begin{proof}
As a first step we bound $t_0$. To this end, 
we write $g(p_0)\coloneqq 2^{1+p_0}\cdot \exp\bigl(\sfrac{1}{\lamWo(\sfrac{2}{p_0})}\bigr)\cdot(2 + \sfrac{1}{e})$
and bound $\lamWo$ by $\log$ with the help of Lemma~\ref{lem:entropy:lambW}, that is 
\begin{equation*}
\frac{\log(\sfrac{1}{p})}{\lamWo(\frac{d+1}{p})}
= \frac{1 + \log(\frac{1}{(d+1)e}) + \log(\frac{d+1}{p})}{\lamWo(\frac{d+1}{p})}
\leq 1 + \frac{1}{(d+1)e}
\leq d\cdot \frac{2 + \sfrac{1}{e}}{d+1}
\eqspace.
\end{equation*}
Together with $p\leq p_0$ and $d\geq 1$ we can bound $t_0$ by
\begin{equation*}
t_0
= \frac{2 (d+1)\cdot 4^{\frac{p}{d+1}}}{e p \cdot \lamWo(\frac{d+1}{p })}\exp\Bigl(\frac{1}{\lamWo(\frac{d+1}{p})}\Bigr) 
\leq \frac{g(p_0)}{e}\cdot d\cdot\frac{\sfrac{1}{p}}{\log(\sfrac{1}{p})}
\eqqcolon f(p)
\eqspace.
\end{equation*}
Since $\beta(t) = \frac{t}{\log(t)}$ is increasing on $[e,\infty)$ according to \eqref{eq:entropy:entropy:temp} and $\sfrac{1}{p}\geq \sfrac{1}{p_0}\geq e$ we get $f(p)\geq f(p_0) = \sfrac{g(p_0)}{e}\cdot d\cdot z_0$, 
where $z_0\coloneqq\frac{\sfrac{1}{p_0}}{\log(\sfrac{1}{p_0})}$. 
Repeating the proof of Lemma~\ref{lem:entropy:binom_bound} together with $d\geq 1$ and $\sfrac{1}{p}\geq\sfrac{1}{p_0}$ yields
\begin{align*}
K_{d,1,p}
= \binom{t_0 + d}{d} \cdot \frac{d+1}{e p }\cdot 4^{\frac{p}{d+1}}
&\leq \binom{f(p_0) + d}{d}\cdot \Bigl(\frac{f(p)}{f(p_0)}\Bigr)^d\cdot d \cdot \frac{2^{1+p_0}}{e}\cdot\sfrac{1}{p}\\
&= \binom{f(p_0) + d}{d}\cdot z_0^{-d}\cdot d \cdot \frac{2^{1+p_0} }{e}\cdot\frac{(\sfrac{1}{p})^{d+1}}{\log^d(\sfrac{1}{p})}
\eqspace.
\end{align*}
Now, we bound the quantities depending on $d$. To this end, we use Lemma~\ref{lem:entropy:binom_behavior} together with Stirling's formula, $f(p_0)\geq 4$, $d\geq 1$, and $1+t\leq e^t$ for $t=\sfrac{d}{f(p_0)}$
\begin{align*}
\binom{f(p_0) + d}{d}\cdot z_0^{-d}\cdot d
&\leq \frac{e^{\sfrac{1}{60}}}{\sqrt{2\pi}}\cdot \Bigl(\frac{1}{d} + \frac{1}{f(p_0)}\Bigr)^{\sfrac{1}{2}}\Bigl(1 + \frac{f(p_0)}{d}\Bigr)^d\Bigl(1 + \frac{d}{f(p_0)}\Bigr)^{f(p_0)}\cdot z_0^{-d}\cdot d\\
&\leq \frac{e^{\sfrac{1}{60}}}{\sqrt{2\pi}}\cdot\Bigl(1 + \frac{e}{g(p_0)\cdot z_0}\Bigr)^{\sfrac{1}{2}}\cdot\Bigl(\frac{1 + \sfrac{g(p_0)}{e}\cdot z_0}{z_0}\Bigr)^d\cdot e^d\cdot d^{\sfrac{1}{2}}
\eqspace.
\end{align*}
Since $C_0=\sfrac{e}{z_0} + g(p_0)$ and the arising quantities that are independent of $p$ and $d$ satisfy
\begin{equation*}
\frac{e^{\sfrac{1}{60}}}{\sqrt{2\pi}}\cdot\Bigl(1 + \frac{e}{g(p_0)\cdot z_0}\Bigr)^{\sfrac{1}{2}}\cdot \frac{2^{1+p_0}}{e} 
\leq 
\frac{e^{\sfrac{1}{60}}}{\sqrt{2\pi}}\cdot\Bigl(1 + \frac{1}{2(2+\sfrac{1}{e})}\Bigr)^{\sfrac{1}{2}}\cdot \frac{2^{1+\sfrac{1}{e}}}{e} 
\approx 0.4239
\leq \sfrac{1}{2}
\end{equation*}
the assertion is proven.
\end{proof}
 

\subsection{Anisotropic Gaussian Kernels}\label{sec:decomposition}


In this section we first provide some general theory about covering numbers of RKHSs and finally prove Theorem~\ref{thm:results:decomposition}. To this end, we introduce some notation. For a fixed bounded kernel $k$ defined on a set $X$ we often consider its restriction to different subsets $Y\subseteq X$. Consequently, we highlight the considered domain by writing $H(Y)$ for the corresponding RKHS and by using the abbreviation $\emb{}{Y}$
for the corresponding embedding $\Id:H(Y)\to\ell_\infty(Y)$.
Recall that $\emb{}{Y}$ is well-defined according to \cite[Lemma~4.23]{StCh2008}. 

\begin{lem}\label{lem:decomposition:sets}
Let $T:Y\to X$ be a mapping between two non-empty sets and $k$ be a bounded kernel on $X$ with RKHS $H(X)$. Then 
\begin{equation}\label{eq:decomposition:transformed_kernel}
k_T(y,y')\coloneqq k(T(y),T(y'))
\eqspace,\qquad\qquad y,y'\in Y,
\end{equation}
defines a bounded kernel on $Y$ with RKHS $H_T(Y)= \bigl\{f\circ T:\ f\in H(X)\bigr\}$ and the corresponding RKHS-norm satisfies 
\begin{equation*}
\|f\circ T\|_{H_T(Y)}\leq \|f\|_{H(X)}
\end{equation*}
for $f\in H(X)$. Moreover, the covering numbers 
satisfy
\begin{equation}\label{eq:decomposition:sets}
\covering[big]{\varepsilon}{\Id:H_T(Y)\to\ell_\infty(Y)} 
\leq \covering[big]{\varepsilon}{\Id: H(X)\to\ell_\infty(X)}
\eqspace,\qquad\qquad \varepsilon>0.
\end{equation}
If, in addition, $T$ is bijective, then equality holds in \eqref{eq:decomposition:sets}.
\end{lem}

\begin{proof}
Let $\Phi:X\to H(X)$ be the canonical feature map of $k$, that is 
$\Phi(x)\coloneqq k(x,\cdot)$ for $x\in X$. Then it is easy to see that $\Phi_T\coloneqq\Phi\circ T$ is a feature map for $k_T$. Consequently, $k_T$ is a kernel on $Y$, and according to \cite[Theorem~4.21]{StCh2008} the RKHS of $k_T$ has the claimed form, the claimed norm inequality is satisfied, and $S_H:H(X)\to H_T(Y)$ defined by $f\mapsto f\circ T$ is a metric surjection, i.e.\ $S_H\ouball{H(X)} = \ouball{H_T(Y)}$. Now, it remains to prove the covering number bounds. To this end, we define the mapping $S_\infty:\ell_\infty(X)\to\ell_\infty(Y)$ by $f\mapsto f\circ T$. 
Using the metric surjectivity of $S_H$
and $\emb{T}{Y}\circ S_H = S_\infty\circ\emb{}{X}$ we get, for $\varepsilon>0$,
\begin{equation*}
\covering{\varepsilon}{\emb{T}{Y}}
=\covering{\varepsilon}{\emb{T}{Y}\circ S_H}
=\covering{\varepsilon}{S_\infty\circ \emb{}{X}}
\eqspace.
\end{equation*}
Since $\|S_\infty f\|_{\ell_\infty(Y)} = \sup_{y\in Y}|f(T(y))| \leq \|f\|_{\ell_\infty(X)}$
is satisfied for all $f\in\ell_\infty(X)$ we have $\|S_\infty\|\leq 1$ and together with \eqref{eq:covering:add_mult} this yields the assertion.
If $T$ is bijective we can exchange the role of $X$ and $Y$ and hence we get the claimed equality.
\end{proof}

\begin{lem}\label{lem:decomposition:partition}
Let $X=X_1\cup X_2$ be the disjoint union of non-empty sets $X_1, X_2$ and $k$ be a bounded kernel on $X$ with RKHS $H(X)$. 
Then for all 
 $\varepsilon>0$ we have
\begin{equation*}
\covering[big]{\varepsilon}{\Id:H(X)\to\ell_\infty(X)} 
\leq \covering[big]{\varepsilon}{\Id:H(X_1)\to\ell_\infty(X_1)}\cdot \covering[big]{\varepsilon}{\Id:H(X_2)\to\ell_\infty(X_2)}
\eqspace.
\end{equation*}
\end{lem}

\begin{proof}
Let $m\coloneqq\covering{\varepsilon}{\emb{}{X_1}}$ and $n\coloneqq\covering{\varepsilon}{\emb{}{X_2}}$. Moreover, choose corresponding $\varepsilon$-nets $f_1,\ldots, f_m\in \ell_\infty(X_1)$ and $g_1,\ldots,g_n\in\ell_\infty(X_2)$. Then for each $i\in\{1,\ldots, m\}$ and each $j\in\{1,\ldots,n\}$ we define
\begin{equation*}
h_{i,j}(x) \coloneqq \begin{cases}
f_i(x), & x\in X_1\\
g_j(x), & x\in X_2\eqspace,
\end{cases}
\qquad \text{for }x\in X.
\end{equation*}
This defines at most $m\cdot n$ different elements of $\ell_\infty(X)$ and it remains to show that $h_{i,j}$ for $i=1,\ldots, m$ and $j=1,\ldots,n$ defines an $\varepsilon$-net of $\cuball{H(X)}$. For $h\in H(X)$ with $\|h\|_{H(X)}\leq 1$ we have $h|_{X_\ell}\in H(X_\ell)$ with $\|h|_{X_\ell}\|_{H(X_\ell)}\leq 1$, for $\ell=1,2$, see Lemma~\ref{lem:decomposition:sets}. Consequently, there is an $i\in\{1,\ldots,m\}$ and a $j\in\{1,\ldots,n\}$ with $\|h|_{X_1} - f_{i}\|_{\ell_\infty(X_1)} \leq \varepsilon$ and $\|h|_{X_2} - g_j\|_{\ell_\infty(X_2)} \leq \varepsilon$, respectively. For this choice of $i$ and $j$ we have
\begin{equation*}
\|h - h_{i,j}\|_{\ell_\infty(X)} = \max\{\|h|_{X_1} - f_{i}\|_{\ell_\infty(X_1)}, \|h|_{X_2} - g_{j}\|_{\ell_\infty(X_2)}\} \leq \varepsilon
\end{equation*}
and hence the assertion is proven.
\end{proof}

So far, we considered bounded kernels on general sets. In the following, we investigate bounded kernels $k:V\times V\to\R$ on a vector space $V$. 
The kernel $k$ is called \emph{translation invariant} along $a\in V$ if
\begin{equation*}
k(v+a, v'+a) = k(v,v')
\end{equation*}
is satisfied for all $v,v'\in V$. In this case the transformation $T(x)\coloneqq x+a$ does not change the kernel, i.e.\ $k=k_T$. Since $T$ is bijective as a mapping $X\to a+X$, Lemma~\ref{lem:decomposition:sets} yields
\begin{equation}\label{eq:decomposition:tansl_invar}
\covering[big]{\varepsilon}{\Id:H(X)\to\ell_\infty(X)} 
= \covering[big]{\varepsilon}{\Id:H(X+a)\to\ell_\infty(X+a)}
\eqspace,\qquad\qquad \varepsilon>0.
\end{equation}
If $k$ is translation invariant along all $a\in U\subseteq V$ for some subspace $U\subseteq V$, then we call $k$ translation invariant along $U$.

\begin{lem}\label{lem:decomposition:general}
Let $(V,\|\cdot\|)$ be a Banach space with complemented subspaces $V_1,V_2\subseteq V$, i.e.\ $V=V_1+ V_2$ and $V_1\cap V_2 = \{0\}$. Moreover, let $X_i\subseteq V_i$ be non-empty subsets, for $i=1,2$, and $k$ be a bounded kernel on $V$. If $k$ is translation invariant along $V_1$ and $X_1$ is relatively compact, then the log-covering numbers satisfy, for $\delta>0$ and $\varepsilon>0$,
\begin{equation*}
\coveringl[big]{\varepsilon}{\Id:H(X_1 + X_2)\to\ell_\infty(X_1 + X_2)}
\leq \covering{\delta}{X_1}\cdot \coveringl[big]{\varepsilon}{\Id:H(\delta\cuball{V_1} + X_2)\to\ell_\infty(\delta\cuball{V_1} + X_2)}
\eqspace.
\end{equation*}
\end{lem}

\begin{proof}
Let us fix some $\varepsilon,\delta>0$ and set $n\coloneqq\covering{\delta}{X_1}$. For a minimal $\delta$-net $x_{1,1},\ldots, x_{1,n}\in V_1$ of $X_1$ we choose a partition $X_{1,1},\ldots, X_{1,n}$ of $X_1$ with $X_{1,i}\subseteq x_{1,i} + \delta\cuball{V_1}$ for all $i=1,\ldots, n$. Since we have chosen a minimal $\delta$-net $X_{1,i}\not=\emptyset$ is satisfied for $i=1,\ldots, n$. Because $X_i\subseteq V_i$, for $i=1,2$, and $V_1,V_2$ are complemented subspaces the sets $X_{1,i}+ X_2$, for $i=1,\ldots,n$, form a partition of $X_1 + X_2$ with $X_{1,i}+ X_2\subseteq x_{1,i} + \delta\cuball{V_1}+ X_2$. A multiple application of Lemma~\ref{lem:decomposition:partition} and an application of Lemma~\ref{lem:decomposition:sets} for $T=\Id$ yield
\begin{equation*}
\coveringl{\varepsilon}{\emb{}{X}} 
\leq \sum_{i=1}^n \coveringl[big]{\varepsilon}{\emb[big]{}{X_{1,i} + X_2}}
\leq \sum_{i=1}^n \coveringl[big]{\varepsilon}{\emb[big]{}{x_{1,i} + (\delta\cuball{V_1})+ X_2}}
\eqspace.
\end{equation*}
Since $k$ is translation invariant along $V_1$, Equation~\eqref{eq:decomposition:tansl_invar} yields the assertion.
\end{proof}

\begin{proof}[Proof of Theorem~\ref{thm:results:decomposition}]
Let $X\subseteq\R^d$ be a bounded subset and $\vecsigma=(\sigma_1,\ldots, \sigma_d)\in(0,\infty)^d$. 
With the notation introduced in \eqref{eq:decomposition:transformed_kernel} the Gaussian kernel then writes as $k_\vecsigma = k_{D_\vecsigma}$. Since the diagonal operator $D_\vecsigma: X\to D_\vecsigma X$ is bijective, Lemma~\ref{lem:decomposition:sets} yields $\coveringl{\varepsilon}{\emb{\vecsigma}{X}} = \coveringl{\varepsilon}{\emb{1}{D_\vecsigma X}}$. Together with Lemma~\ref{lem:decomposition:general} for $\delta=1$, $V_1=\R^d$ (equipped with the Euclidean norm), $V_2=\{0\}$, and $X_1=D_\vecsigma X$, $X_2=\{0\}$ we get the assertion.
\end{proof}

Finally, we present a lemma bounding the covering numbers of convex sets $X\subseteq\R^d$. This result is well-known but we did not find exactly this one in the literature and hence we included a proof for convenience.

\begin{lem}\label{lem:decomposition:volume}
Let $X\subseteq\R^d$ be a convex set and $r_0>0$ such that there is an $a\in\R^d$ with $a + r_0 B_2^d\subseteq X$. Then 
we have
\begin{equation}\label{eq:decomposition:volume}
\covering{2\varepsilon}{X}\leq \frac{\lambda^d(X)}{\lambda^d(B_2^d)} \Bigl(\frac{1}{r_0}+\frac{1}{\varepsilon}\Bigr)^d
\eqspace,\qquad\qquad \varepsilon>0,
\end{equation}
where the covering numbers are with respect to the Euclidean norm and $\lambda^d$ denotes the $d$-dimensional Lebesgue measure.
\end{lem}

\begin{proof}
For this proof we use \emph{packing numbers}, which for $\varepsilon>0$ are defined by
\[ 
\packing{\varepsilon}{X} \coloneqq \max\Bigl\{n\geq 1:\ \exists x_1,\ldots,x_n\in X\text{ with }\|x_i-x_j\|_{\ell_2^d}>2\varepsilon\ \forall i\not=j\Bigr\}
\eqspace.
\]
Recall that $\packing{2\varepsilon}{X} \leq \covering{2\varepsilon}{X} \leq \packing{\varepsilon}{X}$ holds for all $\varepsilon>0$, see e.g.\ \cite[Theorem~IV]{KoTi1961}. Consequently, it is enough to bound $\packing{\varepsilon}{X}$ by the right hand side of \eqref{eq:decomposition:volume}.
For $\varepsilon>0$ we set $n\coloneqq\packing{\varepsilon}{X}$ and choose $x_1,\ldots, x_n\in X$ with $\|x_i - x_j\|_{\ell_2^d}>2\varepsilon$ for all $i\not=j$. Then the sets $x_i + \varepsilon B_2^d$ are disjoint subsets of $X+\varepsilon B_2^d$ and hence
\begin{equation*}
n\varepsilon^d \lambda^d(B_2^d) 
= \lambda^d\Bigl(\bigcup_{i=1}^n \bigl(x_i + \varepsilon B_2^d\bigr)\Bigr) 
\leq \lambda^d(X+\varepsilon B_2^d)
\eqspace. 
\end{equation*}
Since $X$ is convex we have $s_1 X + s_2 X = (s_1+s_2)X$ for $s_1,s_2>0$. 
Together with $r_0 B_2^d \subseteq X - a$ we get
\begin{equation*}
X+\varepsilon B_2^d 
= X + \frac{\varepsilon}{r_0}\cdot r_0 B_2^d 
\subseteq X + \frac{\varepsilon}{r_0}\cdot (X-a)
=\Bigl(1+\frac{\varepsilon}{r_0}\Bigr)X - \frac{\varepsilon}{r_0}a
\eqspace.
\end{equation*}
Both bounds together yield
$n\varepsilon^d \lambda^d(B_2^d) \leq \lambda^d(X) (1+\sfrac{\varepsilon}{r_0})^d$,
which gives the assertion.
\end{proof}



\subsection*{Acknowledgment}
%
The authors thank the International Max Planck Research School for Intelligent Systems (IMPRS-IS)
for supporting Simon Fischer.
%
%


\begin{singlespace}
\bibliographystyle{bibliographystyle}
\bibliography{literatur}
\end{singlespace}


\end{document}